\documentclass[a4paper,oneside,reqno]{amsart}
\pdfoutput=1
\usepackage[utf8]{inputenc}
\usepackage{mathtools,amsmath,amsthm,amssymb}
\usepackage{xcolor}
\usepackage{hyperref}
\usepackage{mathrsfs}

\newtheorem{theorem}{Theorem}

\newtheorem{lemma}[theorem]{Lemma}
\newtheorem{definition}[theorem]{Definition}
\newtheorem{proposition}[theorem]{Proposition}

\newcommand{\ghd}[1]{d_{GH}^{#1}}
\newcommand{\hd}[1]{d_{H}^{#1}}
\newcommand{\id}{\operatorname{id}}
\newcommand{\dist}{d}

\newcommand{\N}{\mathbb N}

\newcommand{\R}{\mathbb R}

\newcommand{\eps}{\varepsilon}

\newcommand{\abs}[1]{\left\lvert #1 \right\rvert}
\newcommand{\der}{\mathrm d}
\newcommand{\ip}[2]{\left\langle#1,#2\right\rangle}
\newcommand{\diam}{\operatorname{diam}}
\newcommand{\sisus}{\operatorname{int}}

\newcommand{\p}{\partial}

\newcommand{\Gr}{\mathcal{G}}
\newcommand{\Hess}{\mathcal H}

\newcommand{\Order}{\mathcal O}

\newcommand{\Pwave}{\textit{P}}
\newcommand{\Swave}{\textit{S}}

\newcommand{\CS}[3]{\Gamma^{#1}_{\phantom{#1}#2#3}}

\newcommand{\Ci}{\mathscr{C}_{\mathrm{JF}}}
\newcommand{\Cii}{\mathscr{C}_{\mathrm{SFF}}}
\newcommand{\Civ}{\mathscr{C}_{\mathrm{diam}}}
\newcommand{\Cv}{\mathscr{C}_{\mathrm{sec}-}}
\newcommand{\Cvi}{\mathscr{C}_{\mathrm{sec}+}}
\newcommand{\Cvii}{\mathscr{C}_{\mathrm{exp}}}
\newcommand{\Cviii}{\mathscr{C}_{\mathrm{dist}}}
\newcommand{\Cix}{C_{1}} 
\newcommand{\Cx}{C_{2}} 
\newcommand{\Cxi}{C_{3}} 
\newcommand{\Cxii}{C_{4}} 
\newcommand{\Cxiii}{C_{\mathrm{b}}} 
\newcommand{\Cxiv}{C_{\mathrm{d}}}
\newcommand{\Cxv}{C_{\mathrm{e}}}
\newcommand{\Cxvi}{C_{\mathrm{h}}}
\newcommand{\Cxvii}{C_{\mathrm{g}}}

\newcommand{\Cxix}{C_{5}} 
\newcommand{\Cxx}{\mathscr{C}_{\mathrm{H1}}}
\newcommand{\Cxxi}{\mathscr{C}_{\mathrm{H2}}}
\newcommand{\Cxxii}{C_{\mathrm{f}}}
\newcommand{\Cxxiii}{C_{\mathrm{i}}}
\newcommand{\Cxxv}{C_{6}} 
\newcommand{\Cxxvi}{C_{7}} 
\newcommand{\Cxxvii}{C_{8}} 
\newcommand{\Cxxviii}{C_{\mathrm{a}}}
\newcommand{\Cxxx}{C_{\mathrm{c}}}

\usepackage{amsaddr}

\title[Stable reconstruction with unknown sources]{Stable reconstruction of simple Riemannian manifolds from unknown interior sources}

\date{\today}

\author{Maarten V. de Hoop}
\address{Department of Computational and Applied Mathematics\\
Rice University\\
6100 Main MS-134, Houston, TX 77005-1892, USA\\
\texttt{mvdehoop@rice.edu}}
\author{Joonas Ilmavirta}
\address{Unit of Computing Sciences\\
Tampere University\\
Kalevantie 4, FI-33014 Tampere University, Finland\\
\texttt{joonas.ilmavirta@tuni.fi}}
\author{Matti Lassas}
\address{Department of Mathematics and Statistics\\
University of Helsinki\\
P.O. Box 68 (Gustaf H\"allstr\"omin katu 2B), FI-00014 University of Helsinki, Finland\\
\texttt{matti.lassas@helsinki.fi}}
\author{Teemu Saksala}
\address{Department of Mathematics\\
North Carolina State University\\
2311 Stinson Drive, Raleigh, NC 27695-8205, USA\\
\texttt{tssaksal@ncsu.edu}}

\begin{document}

\maketitle

\begin{abstract}
Consider the geometric inverse problem: There is a set of delta-sources in spacetime that emit waves travelling at unit speed. If we know all the arrival times at the boundary cylinder of the spacetime, can we reconstruct the space, a Riemannian manifold with boundary? With a finite set of sources we can only hope to get an approximate reconstruction, and we indeed provide a discrete metric approximation to the manifold with explicit data-driven error bounds when the manifold is simple. This is the geometrization of a seismological inverse problem where we measure the arrival times on the surface of waves from an unknown number of unknown interior microseismic events at unknown times. The closeness of two metric spaces with a marked boundary is measured by a labeled Gromov--Hausdorff distance. If measurements are done for infinite time and spatially dense sources, our construction produces the true Riemannian manifold and the finite-time approximations converge to it in the metric sense.
\end{abstract}

\section{Introduction}

We study a geometric inverse problem arising from global seismology.
Multiple sources go off at unknown times and we measure the arrival times, not knowing which arrivals belong to the same event.
The aim is to extract as much information as possible about the planet, modeled here by a simple Riemannian manifold~$M$ of dimension $n=\dim(M)\geq2$, and the sources from this information.

Mathematically, the problem boils down to this:
Suppose we know the boundary of a Riemannian manifold and for some unknown collection of interior points we know the union of their boundary distance graphs, each shifted with an unknown offset.
What can we say about the manifold and these special points from this data?
A more detailed description of the data can be found in Section~\ref{sec:def-thm} below.

The distances between the source points turn out to be determined exactly by the data (Theorem~\ref{thm:A}), but there is no hope of reconstructing the manifold if the source set is finite.
However, the source set can be regarded as a discrete metric space, and this space approximates the Riemannian manifold.
It is not only close as a metric space, but we can also assign approximate boundary points with good accuracy (Theorem~\ref{thm:C23}).
To describe the similarity of two metric spaces ``with the same boundary'', we define a labeled Gromov--Hausdorff distance.
This extends the classical Gromov--Hausdorff distance and compares both the similarity of the metric spaces and the sameness of the boundaries --- with a fixed model space for the boundary.
It is in the sense of this distance that our approximation is good, and the quality of the approximation can be estimated explicitly directly from data.

The sources can be point sources in space time produced by a suitable stochastic process (Proposition~\ref{prop:poisson}), for example.
We present a method to look at the data on the boundary without knowing the manifold and reconstruct an approximate manifold --- and get an explicit error bound on our reconstruction.
This increases the applicability of geometric inverse problems to the messy cases with real data, but we do not attempt to optimize the bounds.
When it comes to explicit geometric estimates, this paper should be seen as a proof of concept.

We assume the source events to be discrete in the spacetime, but their projections to the space can be dense.
In this case the approximate reconstruction is actually perfect, and even the smooth and Riemannian structure is determined (Theorem~\ref{thm:B}).
If measurements are made for finite but increasing time, the approximate reconstructions converge to the true manifold in the labeled Gromov--Hausdorff distance if the sources are spatially dense (Theorem~\ref{thm:D}).

Estimation of the quality of the approximation from data alone is lossy.
To prove that it is not hopelessly so, we prove that by making the source points dense enough the observed density can be made arbitrarily small.
There is an explicit estimate between the true and observed density in both directions (Theorems~\ref{thm:B} and~\ref{thm:reverse}).

We solve a problem stemming from physics by developing geometry rather than by applying \emph{ad hoc} tricks.
These developments include the introduction of the labeled Gromov--Hausdorff distance and various quantitative descriptions of simplicity of a Riemannian manifold.

From a more applied point of view, the key advances are that we introduce of a new and more physically relevant problem and provide quantitative and fully data-driven error estimates.

The precise definitions and statements of results are given in Section~\ref{sec:def-thm}.
We will discuss the results and their context in Section~\ref{sec:discussion}.
The overall strategy of proofs is given in Section~\ref{sec:proof-plan} and the actual proofs follow in subsequent sections.

\subsection*{Acknowledgements}

MVdH was supported by the Simons Foundation under the MATH + X program, the National Science Foundation under grant DMS-1815143, and the corporate members of the Geo-Mathematical Imaging Group at Rice University.
JI was supported by the Academy of Finland (projects 332890 and 336254).
ML was supported by Academy of Finland (projects 284715 and 303754).
TS was supported by the Simons Foundation under the MATH + X program and the corporate members of the Geo-Mathematical Imaging Group at Rice University.
The authors want to thank Peter Caday and Vitaly Katsnelson for useful discussions.

\section{Definitions and theorems}
\label{sec:def-thm}

\subsection{Arrival time data}

Let~$M$ be a Riemannian manifold with boundary.
We denote by $d\colon M\times M\to\R$ its Riemannian distance.
The spacetime $M\times\R$ comes with two natural projections,~$\pi$ onto~$M$ and~$\tau$ onto~$\R$.

Let $S\subset M\times\R$ be a set of sources.
The source $s\in S$ goes off at the point~$\pi(s)$ at the time~$\tau(s)$.
The data measured is the arrival times of signals at the boundary, given as the set
\begin{equation}
Q(S)
=
\{
(x,\tau(s)+d(x,\pi(s))
;\:
x\in\partial M,
s\in S
\}.
\end{equation}
We emphasize that we know the set~$Q(S)$ but not the set~$S$.
We only observe the data as this point set without any labels to tell which arrivals correspond to the same sources.

The arrival time function of a source $s\in S$ is the function $a_s\colon\partial M\to\R$ given by $a_s(x)=\tau(s)+d(\pi(s),x)$.
Let us denote the graph of the function~$a_s$ by $\Gr(a_s)\subset\partial M\times\R$.
The data set can be rewritten as $Q(S)=\bigcup_{s\in S}\Gr(a_s)$.

Let $\phi\colon\partial M_1\to\partial M_2$ be a diffeomorphism between two boundaries.
If $A\subset\partial M_2\times\R$, we define
\begin{equation}
\phi^*A
=
\{
(\phi^{-1}(x),t)
;\:
(x,t)\in A
\}
\subset
\partial M_1\times\R.
\end{equation}
The set~$A$ would typically be the data set~$Q(S_2)$ or a subset thereof.
When there are two manifolds, we decorate all the objects with the subscripts~$1$ and~$2$ when clarity requires so.

A subset of the spacetime $M\times\R$ is said to be discrete when it has no accumulation points.
If~$M$ is compact, it follows that any time slice $M\times[a,b]$ contains only finitely many points of a discrete set.

Our results apply to so-called simple manifolds.
A Riemannian manifold is called simple if it is compact with strictly convex boundary (in the sense of definiteness of the second fundamental form) and the exponential map at every point is a diffeomorphism on its maximal domain of definition.
Consequently there are no conjugate points and any two points on the manifold, including the boundary, are joined by a unique geodesic depending smoothly on the endpoints.

A key concept in our analysis is density.
We say that a subset $A\subset X$ of a metric space~$X$ is $\eps$-dense if the balls $B(a,\eps)$ with $a\in A$ cover~$X$.
Equivalently, every point in~$X$ has a point of~$A$ at a distance below~$\eps$.

We set out to study whether and how a Riemannian manifold~$M$ can be determined by the set~$Q(S)$ and prior knowledge of the boundary~$\partial M$.

\subsection{Precise determination}

We begin by studying what can be determined from the data with no error.

The source sets $S_i\subset M_i\times\R$ can be finite or infinite.
However, the discreteness assumption implies that every $S_i\cap(M_i\times[-T,T])$ is finite and so the source sets are at most countably infinite.
Although the sources are not dense in the spacetime, the source points in space can be dense.
In that case the same data determines the whole manifold up a diffeomorphism.

\begin{theorem}
\label{thm:B}
Let~$M_1$ and~$M_2$ be simple Riemannian manifolds so that the diameter of both manifolds is at most~$\Civ$ and the sectional curvature is bounded from above by $\Cvi>0$ so that
\begin{equation}
\label{eq:diameter-curvature-bound}
\Civ\sqrt{\Cvi}<\pi.
\end{equation}
Let
$\phi\colon\partial M_1\to\partial M_2$
be a Riemannian isometry.
Suppose the two sets $S_i\subset \sisus(M_i)\times\R$ are discrete in $M_i\times\R$.
Let $Q(S_i)\subset\partial M_i\times\R$ be defined as above.

If the set $\pi_1(S_1)\subset M_1$ is dense and $Q(S_1)=\phi^*Q(S_2)$, then
\begin{enumerate}
\item
there is a smooth Riemannian isometry 
$\Phi\colon M_1\to M_2$
so that
$\phi=\Phi|_{\partial M_1}$ and
\item
there is a bijection $\xi\colon S_1\to S_2$ between the sources so that 
\begin{equation}
\label{eq:xi-Phi}
\xi(p,t)
=
(\Phi(p),t)
\end{equation}
for all $(p,t)\in S_1$.
\end{enumerate}
\end{theorem}

The condition~\eqref{eq:diameter-curvature-bound} in most of our theorems always holds when the manifold is \emph{a priori} known to have non-positive sectional curvature.
See Section~\ref{sec:325} for details.

The theorem immediately implies that the source times coincide: $\tau_1(s)=\tau_2(\xi(s))$ for all $s\in S_1$.

The main novelty of Theorem~\ref{thm:B} is in allowing the data set~$Q(S)$ to be a union rather than an indexed collection of graphs of the arrival time functions.
The rest of our results do not have a similar precedent, and our method of proof for Theorem~\ref{thm:B} is new and constructive.
See the discussion in Section~\ref{sec:discussion} below for details.

The next result concerns a situation with a finite amount of sources.
Informally, the next theorem states that knowledge of~$\partial M$ as a smooth manifold and the set~$Q(S)$ determines the distances between the sources, the time differences, and two kinds of distance differences at the boundary.
More formally:

\begin{theorem}
\label{thm:A}
Let~$M_1$ and~$M_2$ be simple Riemannian manifolds and $\phi\colon\partial M_1\to\partial M_2$ a diffeomorphism.
Suppose the two sets $S_i\subset \sisus(M_i)\times\R$ are discrete in $M_i\times\R$ and $\#\pi(S_1)\geq2$.
Let $Q(S_i)\subset\partial M_i\times\R$ be defined as above.

If $Q(S_1)=\phi^*Q(S_2)$, then
\begin{enumerate}
\item
\label{claim:i}
there is a bijection $\xi\colon S_1\to S_2$,
\item
\label{claim:ii}
$
d_1(\pi_1(s),\pi_1(r))
=
d_2(\pi_2(\xi(s)),\pi_2(\xi(r)))
$
for all $s,r\in S_1$,
\item
\label{claim:iii}
$
\tau_1(s)-\tau_1(r)
=
\tau_2(\xi(s))-\tau_2(\xi(r))
$
for all $s,r\in S_1$,
\item
\label{claim:iv}
$
d_1(\pi_1(s),x)-d_1(\pi_1(r),x)
=
d_2(\pi_2(\xi(s)),\phi(x))-d_2(\pi_2(\xi(r)),\phi(x))
$
for all $x\in\partial M_1$ and $s,r\in S_1$,
and
\item
\label{claim:v}
$
d_1(\pi_1(s),x)-d_1(\pi_1(s),y)
=
d_2(\pi_2(\xi(s)),\phi(x))-d_2(\pi_2(\xi(s)),\phi(y))
$
for all $x,y\in\partial M_1$ and $s\in S_1$.
\end{enumerate}
\end{theorem}

While measuring for infinite time gives full uniqueness, it is not realistic.
Therefore we consider measurements for a finite but increasing time and show that there is an approximate reconstruction and in the limit of infinite time it tends to the correct one in a suitable sense.
To be able to state the results, we need to set up a way to compare the true manifold to a discrete approximation.
Before pursuing this direction, we record a result that helps clean up the data by disentangling the union of graphs into separate graphs.
The arrival time functions~$a_s$ are easily verified to be smooth when $\pi(s)\notin\partial M$ and the manifold is simple.

\begin{proposition}
\label{prop:separation}
Let~$M$ be a simple Riemannian manifold.
Suppose $S\subset \sisus(M)\times\R$ is discrete in $M\times\R$.
Then the set~$Q(S)$ and the Riemannian structure of~$\partial M$ determine the set
\begin{equation}
\{
a_s
;\:
s\in S
\},
\end{equation}
of arrival time functions. Furthermore, if $\Omega\subset\partial M\times\R$ is any open set, then~$Q(S)$ determines the set of connected components of the sets $\Gr(a_s)\cap\Omega$, $s\in S$.
\end{proposition}

\subsection{Labeled Gromov--Hausdorff distance}

This subsection is devoted to the metric geometry we will use to compare our approximate reconstructions to the true manifold.

\begin{definition}
Let~$X$ and~$Y$ be compact metric spaces and~$L$ any set (understood as a set of labels).
Let $\alpha\colon L\to X$ and $\beta\colon L\to Y$ be any functions.
We define the labeled Gromov--Hausdorff distance between $(X,\alpha)$ and $(Y,\beta)$ to be
\begin{equation}
\begin{split}
\ghd{L}(X,\alpha;Y,\beta)
&=
\inf
\{
\hd{Z}(f(X),g(Y))
+
\sup_{\ell\in L}
d_Z(f(\alpha(\ell)),g(\beta(\ell)))
;\:
\\&
Z\text{ is a compact metric space},
\\&
f\colon X\to Z
\text{ and }
g\colon Y\to Z
\text{ are isometric embeddings}
\}.
\end{split}
\end{equation}
\end{definition}

Here~$\hd{Z}$ is the Hausdorff distance on the metric space~$Z$.
If the set of labels is empty, then this distance reduces to the usual Gromov--Hausdorff distance; in this context we understand the supremum of an empty set to be zero.

We will be comparing our discrete reconstruction to the manifold~$M$ with labels given by the known set $L=\partial M$, labeled by the inclusion $\iota\colon\partial M\to M$.
To make it meaningful to state that the labeled Gromov--Hausdorff distance is small, we must ensure that this is a well-behaved concept of distance.

The point of this definition is that in addition to getting a good metric approximation of the manifold in the sense of Gromov--Hausdorff distance, we want to get knowledge of how well the known boundary~$\partial M$ sits inside the approximation.
This is what the concept is designed to do.

\begin{proposition}
\label{prop:LGH}
The labeled Gromov--Hausdorff distance is symmetric and satisfies the triangle inequality.
Moreover,
\begin{equation}
\ghd{L}(X,\alpha;Y,\beta)=0
\end{equation}
if and only if there is an isometry $h\colon X\to Y$ so that $h\circ\alpha=\beta$.
\end{proposition}

\subsection{Quantitative simplicity}

We will make approximate reconstructions of the whole manifold, and we can quantify the error of the reconstruction precisely once we have \emph{a priori} bounds on curvature, diameter, and similar geometric properties.

\begin{definition}
\label{def:bdd-geometry}
We say that a manifold~$M$ with boundary has bounded geometry with constants
$\Civ,\Cv,\Cvi,\Cvii,\Ci,\Cii,\Cviii,\Cxx,\Cxxi>0$
if the following properties hold:
\begin{enumerate}
\item
\label{bound:4}
The diameter of~$M$ is at most~$\Civ$.
\item
\label{bound:5}
The sectional curvature is always in the range $[-\Cv,\Cvi]$.
\item
\label{bound:6}
For any $x\in M$ and $\eta_1,\eta_2\in T_xM$ we have
\begin{equation}
\abs{\eta_1-\eta_2}
\leq
\Cvii d(\exp_x\eta_1,\exp_x\eta_2)
\end{equation}
as long as the exponentials are defined.
\item
\label{bound:1}
Every Jacobi field~$J$ with $J(0)=0$ along any unit speed geodesic satisfies
\begin{equation}
\frac{t}{2}\partial_t\abs{J(t)}^2
\leq
\Ci\abs{J(t)}^2
\end{equation}
for all $t>0$ for which the geodesic is defined.
\item
\label{bound:2}
If~$h_1$ and~$h_2$ are, respectively, the first and second fundamental forms of the boundary~$\partial M$, then $h_2\leq \Cii h_1$.
\item
\label{bound:7}
The distances on~$M$ and along~$\partial M$ satisfy
\begin{equation}
d_{M}(x,y)
\leq
d_{\partial M}(x,y)
\leq
\Cviii d_{M}(x,y)
\end{equation}
for all $x,y\in\partial M$.
\item
\label{bound:8}
Take any unit speed geodesic~$\gamma$ on~$M$ and let~$\rho$ be the distance function $\rho(x)=d(x,\gamma(0))$.
For any $t>0$, let $w\in T_{\gamma(t)}M$ be any vector orthogonal to~$\dot\gamma(t)$.
Then the Hessian~$\Hess_\rho$ of~$\rho$ satisfies
\begin{equation}
\ip{w}{\Hess_\rho w}
\geq
\abs{w}^2
(
\Cxx t^{-1}
-
\Cxxi
).
\end{equation}
\end{enumerate}
\end{definition}

We will derive a number of other estimates from these with constants depending on the ones listed here.
See Section~\ref{sec:density} for details.
The constants appearing in the derived estimates are all given in Appendix~\ref{app:const}.
We will refer to the appendix whenever a new constant is introduced in a claim.
All claims are given with simple constants, so there will be a number of relations between the various constants as is clear from the appendix.

It should be noted in part~\ref{bound:8} of the definition that~$w$ is a vector and~$\Hess_\rho w$ is a covector.
The inner product is their duality pairing, which we find most convenient to write as $\ip{w}{\Hess_\rho w}$ instead of $(\Hess_\rho w)(w)$ or $\Hess_\rho(w,w)$.

Having these estimates and constants is not an added assumptions, but merely a quantification of the simplicity of the manifold~$M$.
Our next result justifies that the constants of Definition~\ref{def:bdd-geometry} can well be called constants of quantitative simplicity.

\begin{proposition}
\label{prop:bdd-geom}
Every simple manifold has bounded geometry in the sense of Definition~\ref{def:bdd-geometry} for some constants.
\end{proposition}

\subsection{Approximate determination}

Theorems~\ref{thm:A} and~\ref{thm:B} state what can be determined exactly from our data.
We now turn to studying how well the whole manifold can be reconstructed with finite data.

The source set is as before.
The measurements at the boundary start at time $t=0$ and continue up to time $t=T$.
From this finite amount of data we can then draw an approximate conclusion about the geometry of our manifold.
The error can be quantified very concretely in terms of the geometric bounds.

The only points we know on the manifold are the finite number of sources that have produced a boundary signal on the time interval $[0,T]$ and the \emph{a priori} known boundary itself.
Therefore what we have is a discrete approximation of the smooth Riemannian manifold.
In addition to getting a metric approximation of the manifold itself, we need to approximately embed the known boundary~$\partial M$ into the approximation, and so we will use the Gromov--Hausdorff distance labeled by~$\partial M$.

Let us denote the set of spatial source points by $P=\pi(S)$.
To each $p\in P$ is associated a boundary distance function $r_p\colon\partial M\to\R$ given by $r_p(x)=d(x,p)$.
Let $c(p)\subset\partial M$ be the set of critical points of this function.
By compactness $\# c(p)\geq2$.

Consider a simple manifold of bounded geometry in the sense of Definition~\ref{def:bdd-geometry}.
At $x\in c(p)$ let $\lambda(p,x)$ be the smallest eigenvalue of the Hessian of~$r_p$ at~$x$.
This requires that~$\partial M$ is known as a Riemannian manifold, not just as a smooth one.
For $y\in c(p)$ we define
\begin{equation}
\label{eq:E1}
E(p,y)
=
\begin{cases}
\frac{\Ci}{\lambda(p,y)-\Cii}, & \text{when }\lambda(p,y)>\Cii\\
\infty, &\text{otherwise}
\end{cases}
\end{equation}
and for all $x\in\partial M$
\begin{equation}
\label{eq:E2}
E(x)
=
\inf_{p\in P,y\in c(p)}
\left(E(p,y)+d_{\partial M}(x,y)\right).
\end{equation}
Finally, we let
\begin{equation}
\label{eq:E3}
E
=
\sup_{x\in\partial M} E(x).
\end{equation}
These three types of~$E$s will be all determined by the data and the constants of bounded geometry.
Observe that if the Hessians of the distance functions are never large enough, all these~$E$s may well be infinite. 

For any $p,q\in P$ and $r,s>0$ we define the lentil
\begin{equation}
L^{p,q}_{r,s}
=
B(p,r)\cap B(q,s).
\end{equation}
The lentil is said to have thickness
\begin{equation}
\delta^{p,q}_{r,s}
=
r+s-d(p,q).
\end{equation}
One can check that on a simple manifold the thickness is the length of the segment of the geodesic~$\gamma_{p,q}$ between the two points that lies within the lentil, provided that $r,s<d(p,q)$.
To see this, observe that a point $\gamma_{p,q}(t)$ is in the lentil if and only if $t<r$ and $d(p,q)-t<s$.

Our inversion procedure starts with preprocessing the data.
Proposition~\ref{prop:separation} gives on aspect of it, and the next one says that the derived quantities introduced above are determined by the data.

\begin{proposition}
\label{prop:E&L}
Let~$M_1$ and~$M_2$ be simple Riemannian manifolds.
Let $\phi\colon\partial M_1\to\partial M_2$ be a smooth isometry.
Suppose that the two sets $S_i\subset \sisus(M_i)\times\R$ are discrete in $M_i\times\R$ and $\#\pi(S_1)\geq2$.
Let $Q(S_i)\subset\partial M_i\times\R$ be defined as above.
Let $\xi\colon S_1\to S_2$ be the bijection of Theorem~\ref{thm:A} and let $\Xi\colon P_1\to P_2$ be the map for which $\Xi\circ\pi_1=\pi_2\circ\xi$.
Then:
\begin{enumerate}
\item
The critical point sets correspond via $c(\Xi(p))=\phi(c(p))$ for all $p\in P_1$. The boundary distance quantities correspond via
$E_1(p,y)=E_2(\Xi(p),\phi(y))$ for all  $p\in P$ and $y \in c(p)$. Moreover
and
$E_1(x)=E_2(\phi(x))$
for all $x\in\partial M_1$ and
$E_1=E_2$.
\item
If we fix any $\eps_1>0$ and define
\begin{equation}
\label{eq:Gamma}
\Gamma_i
=
\{
p\in P_i
;\:
E_i(p,y)<\eps_1\text{ for some }y \in c(p)
\}
,
\end{equation}
then $\Gamma_2=\Xi(\Gamma_1)$.
\item
Fix any $\eps_1>0$ and $\delta>0$.
Define~$\Gamma_i$ for $i=1,2$ as above.
Take any $x_1,y_1\in\Gamma_1$ so that $d(x_1,y_1)>\delta$, and pick any $r\in(\delta,d_1(x_1,y_1))$.
Set $s=d_1(x_1,y_1)-r+\delta$ and 
let
\begin{equation}
L_1
=
L^{x_1,y_1}_{r,s}\subset M_1
\quad\text{and}
\quad
L_2
=
L^{x_2,y_2}_{r,s}\subset M_2,
\end{equation}
where $x_2=\Xi(x_1)$ and $y_2=\Xi(y_1)$.
Then
\begin{equation}
\Xi(L_1\cap P_1)
=
L_2\cap P_2
\end{equation}
and so
\begin{equation}
L_1\cap P_1\neq\emptyset
\quad\text{if and only if}\quad
L_2\cap P_2\neq\emptyset
.
\end{equation}
\end{enumerate}
\end{proposition}

Now that we know the data to determine the relevant auxiliary quantities uniquely, we move to working on a single manifold.
The next theorem states that under known bounds on the geometry and metrically known boundary one can estimate directly, from the data, how dense the source points are. For our purposes this is to assign approximate boundary points, and estimate the metric distance of an approximate reconstruction from the true manifold.

In practice, the set~$P$ in the next theorem stands for the spatial source set~$\pi(S)$ as above, but the claim is valid for any subset of~$\sisus(M)$.

\begin{theorem}
\label{thm:C23}
Let~$M$ be simple a Riemannian manifold with the constants of Definition~\ref{def:bdd-geometry} (by Proposition~\ref{prop:bdd-geom}) satisfying the condition~\eqref{eq:diameter-curvature-bound}, and let $P\subset\sisus(M)$ be any 
set.
This theorem uses the constants given in
equations~\eqref{eq:c9}, \eqref{eq:c10}, \eqref{eq:c11}, and~\eqref{eq:c12}.

Fix any $\eps_1>0$ and define the set $\Gamma\subset P$ as in~\eqref{eq:Gamma}.
Set $\eps_2=\eps_1+E$ and $\delta=\Cix\eps_2$.
Suppose that for every $x,y\in\Gamma$ with $d(x,y)>\delta$ and $r\in(\delta,d(x,y))$ we have
\begin{equation}
\label{eq:lentil-intersect-C}
L^{x,y}_{r,d(x,y)-r+\delta}\cap P\neq\emptyset.
\end{equation}
Then:
\begin{enumerate}
\item
\label{claim:2}
The 
set~$P$ is $\eps$-dense in~$M$ with
\begin{equation}
\label{eq:eps<eps2+reps2}
\eps
=
\Cx\eps_2 
+
\Cxi\sqrt{\eps_2}.
\end{equation}
\item
\label{claim:3}
Given any $\alpha\colon\partial M\to P$ so that
\begin{equation}
\label{eq:alpha-choice}
E(x)
+\eps_1
\geq
\inf_{y\in c(\alpha(x))}
\left(E(\alpha(x),y)+d_{\partial M}(x,y)\right)
\end{equation}
for all $x \in \p M$,
the pair $(P,\alpha)$ is close to $(M,\iota)$, with the inclusion $\iota\colon\partial M\to M$, in the sense of the labelled Gromov--Hausdorff distance:
\begin{equation}
\label{eq:ghd<eps2+reps2}
\ghd{\partial M}(P,\alpha;M,\iota)
\leq
\Cxii\eps_2
+
\Cxi\sqrt{\eps_2}.
\end{equation}
\end{enumerate}
\end{theorem}

A map~$\alpha$ satisfying~\eqref{eq:alpha-choice} will be constructed explicitly from the data.
This is explained in the proof of the theorem in Section~\ref{sec:thm-C-pf}.

The first part of this theorem says that under an additional curvature assumption the quantities defined by the data (due to Proposition~\ref{prop:E&L}) give an explicit estimate for the density of the source points.
The second part is an estimate of the labeled Gromov--Hausdorff distance between the approximating discrete space~$P$ with the approximate boundary described by~$\alpha$ and the true manifold~$M$ with boundary~$\partial M$.
The key assumption is that all lentils of thickness~$\delta$, which is proportional to~$\eps_2$, contain a source point, and by Proposition~\ref{prop:E&L} this is determined uniquely by the data.

The quality of the estimates given in claim~\ref{claim:2} of Theorem~\ref{thm:C23} depend on the choice of~$\eps_1$.
When faced with a practical reconstruction task, one ought to vary~$\eps_1$ and see which choice gives the best error bound~$\eps_2$ while still having non-empty lentils.
The only parameter we choose is~$\eps_1$; the second parameter~$\eps_2$ which plays a more central role is based on the observed quantity~$E$ of~\eqref{eq:E3}.
Heuristically,~$\eps_1$ describes how close the chosen source points are to the boundary and~$\eps_2$ how far from any given boundary point these near-boundary sources are.
These constants are then used to establish estimates for interior density of the sources.

The estimates of claim~\ref{claim:2} of Theorem~\ref{thm:C23} are vacuous if $\eps_2=\infty$.
For the theorem to be meaningful, we need to ensure that for some~$\eps_1$ we indeed have $\eps_2<\infty$.
To be able to get an arbitrarily good estimate for density and metric distance, we also need to ensure that~$\eps_2$ will become as small as we like if the source set $P\subset M$ is dense enough.
This is a reverse density estimate:
If the source set is very dense, it will appear dense by our estimates.

\begin{theorem}
\label{thm:reverse}
Let~$M$ be a simple manifold of bounded geometry
with the constants of Definition~\ref{def:bdd-geometry} (by Proposition~\ref{prop:bdd-geom}) satisfying the condition~\eqref{eq:diameter-curvature-bound}
and $S\subset\sisus(M)\times\R$ a discrete source set.
In this theorem we use the constants of~\eqref{eq:c19}, \eqref{eq:c25}, \eqref{eq:c26}, \eqref{eq:c27}, \eqref{eq:c12}, and~\eqref{eq:c11}.

Take any $\eps>0$.
If $P=\pi(S)\subset M$ is $\hat\eps$-dense with
\begin{equation}
\label{eq:hateps}
\hat\eps
=
\min
\left(
\Cxxv
,
\Cxxvi\eps
,
\Cxxvii\eps^2
\right)
,
\end{equation}
and $\eps_1=\Cxix\hat\eps$,
then:
\begin{enumerate}
\item
The quantity $\eps_2=\eps_1+E$ as used in Theorem~\ref{thm:C23} satisfies
\begin{equation}
\label{eq:eps2<eps}
\Cxii\eps_2
+\Cxi\sqrt{\eps_2}
<
\eps.
\end{equation}
\item
Denote $\delta=\Cix\eps_2$.
For any $x,y\in P$ that are more than distance~$\delta$ apart and any $r\in(\delta,d(x,y))$ the lentil
$L^{x,y}_{r,d(x,y)-r+\delta}$ contains an element of~$P$.
\item
The set~$P$ satisfies the assumptions of Theorem~\ref{thm:C23} for the given choice of~$\eps_1$; especially the lentil intersection property~\eqref{eq:lentil-intersect-C} holds.
Therefore 
one can conclude from the data set~$Q(S)$ and the geometric constants alone that~$P$ is $\eps$-dense and the discrete approximation is $\eps$-good.
\end{enumerate}
\end{theorem}

Choosing~$\eps_1$ correctly is important for Theorem~\ref{thm:reverse}.
If~$\eps_1$ is too large, the resulting estimates are not tight enough.
If it is too small, then too few boundary sources are identified and not all lentils contain source points.

\subsection{Convergence of approximations}

The next theorem treats the case of gradually acquired and interpreted data.
Assuming the sources are eventually dense, a finite amount of data determines an approximate manifold.
The data itself gives an estimate for how good the approximation is and the approximations converge to the true manifold with boundary in the labeled Gromov--Hausdorff sense.

\begin{theorem}
\label{thm:D}
Let~$M$ be a simple Riemannian manifold with geometry bounded in the sense of Definition~\ref{def:bdd-geometry} satisfying the condition~\eqref{eq:diameter-curvature-bound}.
Let $S\subset\sisus(M)\times\R$ be a countable discrete set of sources so that the projection $\pi(S\cap(M\times[0,\infty)))\subset M$ is dense.

Then for each $T>0$ the boundary $\p M$ as a Riemannian manifold and the set
\begin{equation}
\label{eq:Q(S,T)}
Q(S,T)
=
Q(S)\cap(\partial M\times[0,T])
\end{equation}
determine a finite metric space~$M_T$ and a map $\alpha_T\colon\partial M\to M_T$ so that
\begin{equation}
\lim_{T\to\infty}
\ghd{\partial M}(M_T,\alpha_T;M,\iota)
=
0
.
\end{equation}
Thus $(M_T,\alpha_T)\to(M,\iota)$ as $T\to\infty$ in the labeled Gromov--Hausdorff sense.
\end{theorem}

Finally, we point out that if the geometric assumptions of Theorem~\ref{thm:D} are satisfied, then the assumptions on the sources are almost certainly satisfied in a simple stochastic model.
That is, for almost any source pattern a finite amount of data determines an approximation and the approximation converges to the true solution with increasing data in a manner quantifiable from data itself.

We say that two measures are uniformly comparable when they are absolutely continuous with respect to each other and the Radon--Nikodym derivatives both ways are bounded.
Let~$\nu$ be the product of the volume measure of~$M$ and the Lebesgue measure of~$\R$, and let~$\mu$ be a measure on $M\times\R$ uniformly comparable with~$\nu$.
A homogeneous Poisson point process on the spacetime $M\times\R$ with intensity $\lambda>0$ is such a point process that the number of points in the set $A\subset M\times\R$ follows the Poisson distribution with parameter~$\lambda\bar\mu(A)$.

\begin{proposition}
\label{prop:poisson}
Let~$M$ be a simple Riemannian manifold with bounded geometry satisfying~\eqref{eq:diameter-curvature-bound}.
Let~$\mu$ be a measure on $M\times\R$ which is uniformly comparable with the natural product measure~$\nu$.
Suppose the set $S\subset M\times\R$ is given by a homogeneous Poisson point process with intensity $\lambda>0$ and the measure~$\mu$.

Then almost surely:
\begin{enumerate}
\item $S$ is countable and discrete,
\item the projection $\pi(S\cap(M\times[0,\infty)))$ is contained in~$\sisus(M)$ and dense in~$M$, and therefore
\item the conclusions of Theorem~\ref{thm:D} hold.
\end{enumerate}
\end{proposition}

That is, if the events $S\subset M\times\R$ are given by a Poisson process, then almost surely the reconstruction from data for finite but increasing time converges to the true solution as time increases.

\section{Discussion}
\label{sec:discussion}

\subsection{Geometric inverse problems and seismology}

An elastic body --- e.g. a planet --- can be modeled as a manifold, where distance is measured in travel time:
The distance between two
points is the shortest time it takes for a wave to go from one point
to the other.
If the material is isotropic or elliptically anisotropic, then this
elastic geometry is Riemannian.
However, this sets a very stringent
assumption on the stiffness tensor describing the elastic system, and
Riemannian geometry is therefore insufficient to describe the
propagation of seismic waves in the Earth.
If no structural
assumptions on the stiffness tensor apart from the physically
necessary symmetry and positivity properties are made, this leads
necessarily modeling the planet by a Finsler manifold as was explained in~\cite{dehoop2020determination}. 

An isotropically elastic medium carries pressure (\Pwave) and shear (\Swave) wave speeds that are conformally Euclidean metrics.
Of these two the \Pwave-waves are faster~\cite{cerveny2005seismic}.
In order to be true to the isotropic elasticity we should measure both \Pwave- and \Swave-wave arrivals.
As we have many sources going off in different locations and we do not know which arrival is related to which source we cannot \emph{a priori} catalog arrivals of different wave types.
We simplify this aspect of the problem by disregarding polarizations and considering only one type of wave.

In the two following subsections we review some seismologically relevant geometric inverse problems on Riemannian and Finsler manifolds.
We also draw relations between the mathematical data, considered in this paper, and some real world seismic measurements.
In the final subsection we discuss our contributions to field of inverse problems and possible future research directions suggested by the results of this paper.

\subsubsection{Geometric inverse problems on Riemannian manifolds}

The problem of determination of the isometry type of a compact Riemannian manifold from its \emph{boundary distance data}
\begin{equation}
\label{eq:BDD}
\{r_p\colon \p M \to (0,\infty);\: p \in \sisus(M)\}
\end{equation}
was introduced for the first time in~\cite{kurylev1997multidimensional}. The reconstruction of the smooth atlas on the manifold and the metric tensor in these coordinates was originally considered in~\cite{Katchalov2001}. In contrast to the paper at hand these earlier results do not need any extra assumption for the geometry, but have a complete data in the sense that the distance function~$r_p$ to the boundary is known for any interior source point $p\in \sisus(M)$.

The problem of boundary distance data is related to many other geometric inverse problems. For instance, it is a crucial step in proving uniqueness for Gel'fand's inverse boundary spectral problem~\cite{Katchalov2001}.
Gel'fand's problem concerns the question whether the data
\begin{equation}
(\p M, (\lambda_j, \p_\nu \phi_j|_{\p M})_{j=1}^\infty)
\end{equation}
determine $(M,g)$ up to isometry, when $(\lambda_j, \phi_j)$ are the Dirichlet eigenvalues and the corresponding $L^2$-orthonormal eigenfunctions of the Laplace--Beltrami operator.
Belishev and Kurylev provide an
affirmative answer to this problem in~\cite{belishev1992reconstruction}.

In~\cite{katsuda2007stability} the authors studied a question of approximating a Riemannian manifold under the assumption: For a finite set of receivers $R\subset \p M$ one can measure the travel times $\{r_p(z);\: z \in R\}$ for any  point $p \in \sisus(M)$.
In this problem arrival times (with known initial times) are measured from all interior sources at a finite number of boundary points, whereas in the present paper we measure arrival times (with unknown initial times) from a finite number of interior sources at all boundary points.
The authors construct an approximate finite metric space ~$M_\eps$, with the same cardinality as the receiver set~$R$, and show that if the Hausdorff distance of~$R$ and~$\p M$ is less than $\eps>0$ then the Gromov--Hausdorff distance of~$M$ and~$M_\eps$ is proportional to some positive power of~$\eps$.
In contrast to the present paper an independent travel time measurement was made for any interior source point.

We recall that for a source $s\in S$ we do not know the initial time~$\tau(s)$, but due to Proposition~\ref{prop:separation} we can recover the arrival times $a_s(z)=d(z,\pi(s))+\tau(s)$ for each $z\in \p M$.
Taking the difference of the arrival times one obtains a boundary distance difference function
\begin{equation}
D_{\pi(s)}(z_1,z_2):=d(\pi(s),z_1)-d(\pi(s),z_2)
\end{equation}
for all $z_1,z_2 \in \p M$,
and this is independent of the initial time.
In~\cite{LaSa}
it is shown that if $U\subset N$ is a compact subset of a closed
Riemannian manifold $(N,g)$ and $\sisus(U)\neq \emptyset$, then
\emph{distance difference data} 
\begin{equation}
((U,g|_{U}), \{D_x\colon U\times U
\to \R \:| \:x \in N\})
\end{equation}
determine $(N,g)$ up to an isometry.
This
result was recently generalized to complete Riemannian manifolds~\cite{ivanov2018distance} and for Riemannian manifolds with
boundary~\cite{de2018inverse, ivanov2020distance}.
These results require sources at all interior points, unlike the ones in the present paper.

If the sign in the definition of the distance difference functions is
changed, we arrive at the distance sum functions,
\begin{equation}
D^+_x(z_1,z_2)=d(z_1,x)+d(z_2,x)
\end{equation}
for all $x\in M$ and $z_1,z_2\in \p M$.
These functions give the lengths of the broken geodesics, that is, the
union of the shortest geodesics connecting~$z_1$ to~$x$ and the
shortest geodesics connecting~$x$ to~$z_2$.
Also, the gradients of $D^+_x(z_1,z_2)$ with respect to~$z_1$ and $z_2$ give the velocity
vectors of these geodesics.
The inverse problem of determining the
manifold $(M,g)$ from the \emph{broken geodesic data}, consisting of
the initial and the final points and directions, and the total length,
of the broken geodesics, has been considered in~\cite{kurylev2010rigidity}. In~\cite{kurylev2010rigidity} the authors
show that broken geodesic data determine the boundary distance data
and use then the results of \cite{Katchalov2001, kurylev1997multidimensional} to prove that the broken geodesic data determine the Riemannian manifold up to an isometry. 
A different variant of broken geodesic data was recently considered in~\cite{meyerson2020stitching}.

The Riemannian wave operator is a globally hyperbolic
linear partial differential operator of real principal
type.
Therefore, the Riemannian distance function and the propagation
of a singularity initiated by a point source in space time are related
to one another.
We let~$u$ be the solution of the Riemannian wave equation with a point
source $s \in S$.
In \cite{duistermaat1996fourier, greenleaf1993recovering} it is shown
that the image, $\Lambda$, of the wavefront set of~$u$, under the
musical isomorphism $T^\ast M \ni (x,\xi) \mapsto (x,g^{ij}(x)\xi_i) \in TM$, coincides
with the image of the unit sphere~$S_{\pi(s)}M$ at~$\pi(s)$ under the geodesic
flow of~$g$. Thus $\Lambda \cap \p(S M)$, where~$SM$ is the unit sphere bundle of $(M,g)$, coincides with the exit
directions of geodesics emitted from~$\pi(s)$.
In~\cite{lassas2018reconstruction} the authors show that if $(M,g)$ is a
compact smooth non-trapping Riemannian manifold with smooth strictly
convex boundary, then generically the \emph{scattering data of point 
sources} $(\p M, R_{\p M}(M))$ determine $(M,g)$ up to
isometry.
Here, $R_{\p M}(x) \in R_{\p M}(M)$ for $x \in M$ stands for
the collection of tangential components to boundary of exit directions
of geodesics from~$x$ to~$\p M$.

If~$M$ is an open subset of a complete Riemannian manifold $(N,g)$, then one more important data set, given by the wavefront set of~$u$, is related to a \emph{generalized spheres} of radius $r>0$, that is given by the formula
\begin{equation}
\begin{split}
S(p,r)
\coloneqq
\{\exp_p(v)\in N;\:
&v\in T_pM \text{, } \|v\|_g=r,
\\&
\text{and $\exp_p$ is not singular at~$v$}\}.
\end{split}
\end{equation}
In~\cite{deHoop1} the authors show that the \emph{spherical surface data}
\begin{equation}
\{S(q,t)\cap N \setminus M;\: q  \in M \text{ and } t>0 \}
\end{equation}
determine the universal cover space of~$N$. If a generalized sphere $S(p,r)$ is given the authors show that there exists a specific coordinate structure in a neighborhood of any maximal normal geodesic to $S(p,r)$ such that in these coordinates metric tensor~$g$ can can be determined. However this does not determine~$g$ globally. The authors provide an example of two different metric tensors which produce the same spherical surface data. 

A classical geometric inverse problem, that is closely related to the distance functions, asks:
Does the Dirichlet-to-Neumann mapping of a Riemannian wave operator determine a Riemannian manifold up to isometry?
For the full boundary data this problem was solved originally in~\cite{belishev1992reconstruction} using the Boundary control method. Partial boundary data questions have been studied for instance in \cite{lassas2014inverse, milne2019codomain}.
Recently~\cite{kurylev2018inverse} extended these results for connection Laplacians.
Lately also inverse problems related to non-linear hyperbolic equations have been studied extensively  \cite{kurylev2014inverse, lassas2018inverse, wang2016inverse}.
For a review of inverse boundary value problems for partial differential equations see \cite{LassasICM2018, uhlmann1998inverse}.

Maybe the most studied geometric inverse problem formulated with the distance functions is the Boundary rigidity problem. This problem asks:  Does the \emph{boundary distance function}, that gives a distance between any two boundary points, determine $(M,g)$ up to an isometry? In an affirmative case $(M,g)$ is said to be boundary rigid. For a general Riemannian manifold the problem is false: Suppose the manifold contains a domain with very slow wave speed, such that all the geodesics starting and ending at the boundary avoid this domain. Then in this domain one can perturb the metric in such a way that the boundary distance function does not change. It was conjectured in~\cite{michel1981rigidite} that for all compact simple Riemannian manifolds the answer is affirmative. In two dimensions it was solved in~\cite{pestov2005two}. For higher dimensional case the problem is still open, but different variations of it has been considered for instance in \cite{burago2010boundary, croke1991rigidity, stefanov2016boundary, stefanov2017local}. 

A general feature of geometric inverse problems related to seismology is that each point source is analyzed in isolation and there is a source in every point on the manifold.
We drop both of these assumptions to step towards a physically more accurate model.
The source set can be finite and the data does not \emph{a priori} make a distinction between different events.
This distinction is achieved by Proposition~\ref{prop:separation}.
Another novelty is that our discrete approximation is quantitatively stable in the sense of the labeled Gromov--Hausdorff distance.
All our estimates are given by the data and some geometric constants and do not depend on other \emph{a priori} knowledge of the manifold~$M$.

\subsubsection{Related geometric inverse problems on Finsler manifolds}

In~\cite{dehoop2020determination} we studied the recovery of a compact Finsler manifold from its boundary distance data. In contrast to earlier Riemannian results \cite{Katchalov2001, kurylev1997multidimensional} the data only determines the topological and smooth structures, but not the global geometry. The Finsler function $F\colon TM \to [0,\infty)$ can be however recovered in a closure of the set $G(M,F)\subset TM$, that consists of points $(p,v)\in TM$ such that the corresponding geodesic~$\gamma_{p,v}$ is distance minimizing to the terminal boundary point. We also showed that if the set $TM \setminus G(M,F)$ is non-empty then any small perturbation of~$F$ in this set leads to a Finsler metric whose boundary distance data agrees with the one of~$F$. If $G(M,F)=TM$, then the boundary distance data determines $(M,F)$ up to a Finsler isometry. For instance the isometry class of any simple Finsler manifold is determined by this data.
The same is not true if only the boundary distance function is known~\cite{ivanov2013local}.
Thus a simple Finsler manifold is never boundary rigid.
In~\cite{de2020foliated} we utilized the main result of~\cite{dehoop2020determination} and generalized the result of~\cite{kurylev2010rigidity}, about the broken geodesic data, on reversible Finsler manifolds, satisfying a convex foliation condition. 

Although simple Finsler manifolds are not boundary rigid there are results considering their rigidity questions for some special Finsler metrics.
For instance it was shown in~\cite{monkkonen2020boundary} that Randers metrics $F_i=G_i+\beta_i$ indexed with $i \in \{1,2\}$ with simple and boundary rigid Riemannian norm $G_i(x,v)=\sqrt{g_{ij}(x)v^iv^j}$ and closed one-form~$\beta_i$, have the same boundary distance function if and only if $G_1=\Psi^{\ast}G_2$ for some boundary fixing diffeomorphism $\Psi\colon M \to M$ and $\beta_1-\beta_2=\mathrm d \phi$ for some smooth function~$\phi$ vanishing on~$\p M$. It is worth of mentioning that analogous results have been presented earlier on a Riemannian manifold in the presence of a magnetic field \cite{dairbekov2007boundary, AsZh}.

The sphere data described in connection to the Riemannian results of~\cite{deHoop1} above has also been studied on Finsler manifolds~\cite{Finsler_Dix}.
Knowledge of the spheres uniquely determine the fundamental tensor and the curvature operator along any geodesic passing through the known domain, but in contrast to Riemannian geometry this information is insufficient for a full reconstruction of the universal cover.

\subsubsection{Geophysical literature}
In fact, the inverse problem considered here can be directly related to seismology. The sources correspond with microseismic events with unknown locations and origin times while the metric corresponds with the wave speed. Here, we consider only one wave speed associated with elastic \Pwave-waves, but incorporating a second wave speed associated with elastic \Swave-waves is quite straightforward. In the past decade there has been extensive research on the joint recovery of wave speed and event locations, with mining, geothermal and hydraulic-fracturing applications of induced seismicity but also in studies of the crustal structure. Here, we present a comprehensive analysis providing fundamental insight in the feasibility of succeeding in this, while focusing on recovering the wave speed. We use geometric data rather than wave fields; in applications, the extraction of these have been routine, for example, with template matching~\cite{FeenstraRoeckeretal-2016}.

Initial empirical studies \cite{JanskyPlickaEisner-2010, BliasGrechka-2013} assumed a simple wave speed model varying in one coordinate only and one or two strings of receivers (in boreholes) penetrating the manifold. Strategies have been broadly based on optimization, such as variations of gradient descent \cite{ShekarNath-2015, IgoninInnanen-2018}, with intertwining updating~\cite{MichelTsvankin-2017} and a neighborhood algorithm~\cite{Tanetal-2018} employing different optimization criteria \cite{WittenShragge-2017, Songetal-2019} or via intermediate (approximate) interior wave-field recovery~\cite{SongAlkhalifah-2019}. A basic statistical, Bayesian framework has been developed in parallel~\cite{ZhangRectorNava-2017}.

\subsection{Discussion of technical assumptions}

In this subsection we present a discussion of many technical assumptions we made in our theorems.
Our focus here is on the key phenomenology of having multiple sources with unknown times and we have chosen to simplify other aspects of the problem.

\subsubsection{Density of source points}

We assume in all our results that the sources are discrete in the spacetime $M\times\R$.
This is physically reasonable and also practical.
If the sources were dense or even accumulated somewhere, Proposition~\ref{prop:separation} would become far more complicated.

However, there is little hope of full uniqueness with finitely many source points.
The sources can be discrete in spacetime but dense in space, and this is indeed the setting of Theorems~\ref{thm:B} and~\ref{thm:D} in which the Riemannian manifold is determined uniquely either from full time data or asymptotically from increasing time data.

\subsubsection{Full data and convexity}

All our results concern full data in the sense that arrivals are recorded on the whole boundary~$\partial M$.
The full boundary is crucial for Proposition~\ref{prop:separation} which disentangles the data into a collection of graphs, and we make so heavy use of differential tools on the boundary throughout the paper that a discrete subset of the boundary is beyond current reach.

Let us construct an explicit example of a surface~$M$ and a subset $\Gamma\subset\partial M$ so that our results fail with data recorded only on~$\Gamma$.
Every pair of points on a smooth compact Riemannian manifold with boundary is connected by a $C^1$-smooth distance minimizing curve~\cite{alexander1981geodesics}.
We choose our a manifold to be the horseshoe-shaped domain of Figure~\ref{Fi:f_p}.

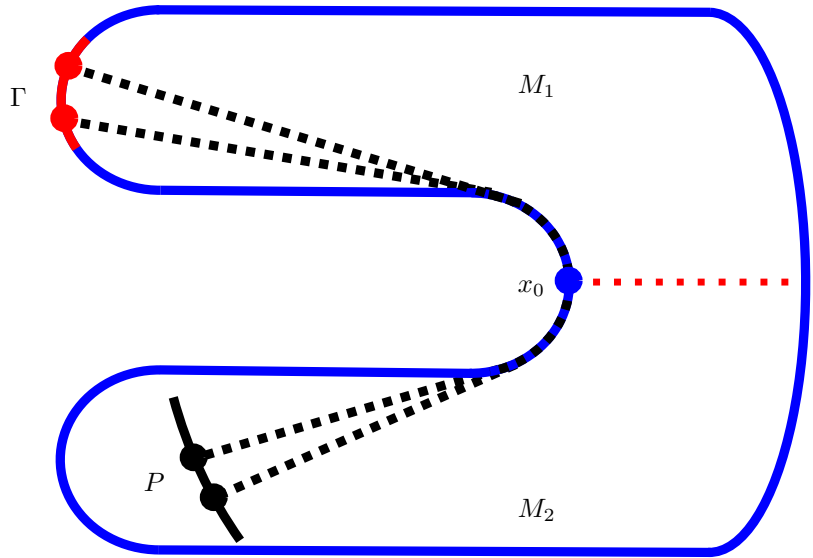
\begin{figure}[ht]
\begin{picture}(300,200)
  \put(0,0){\includegraphics[width=10cm]{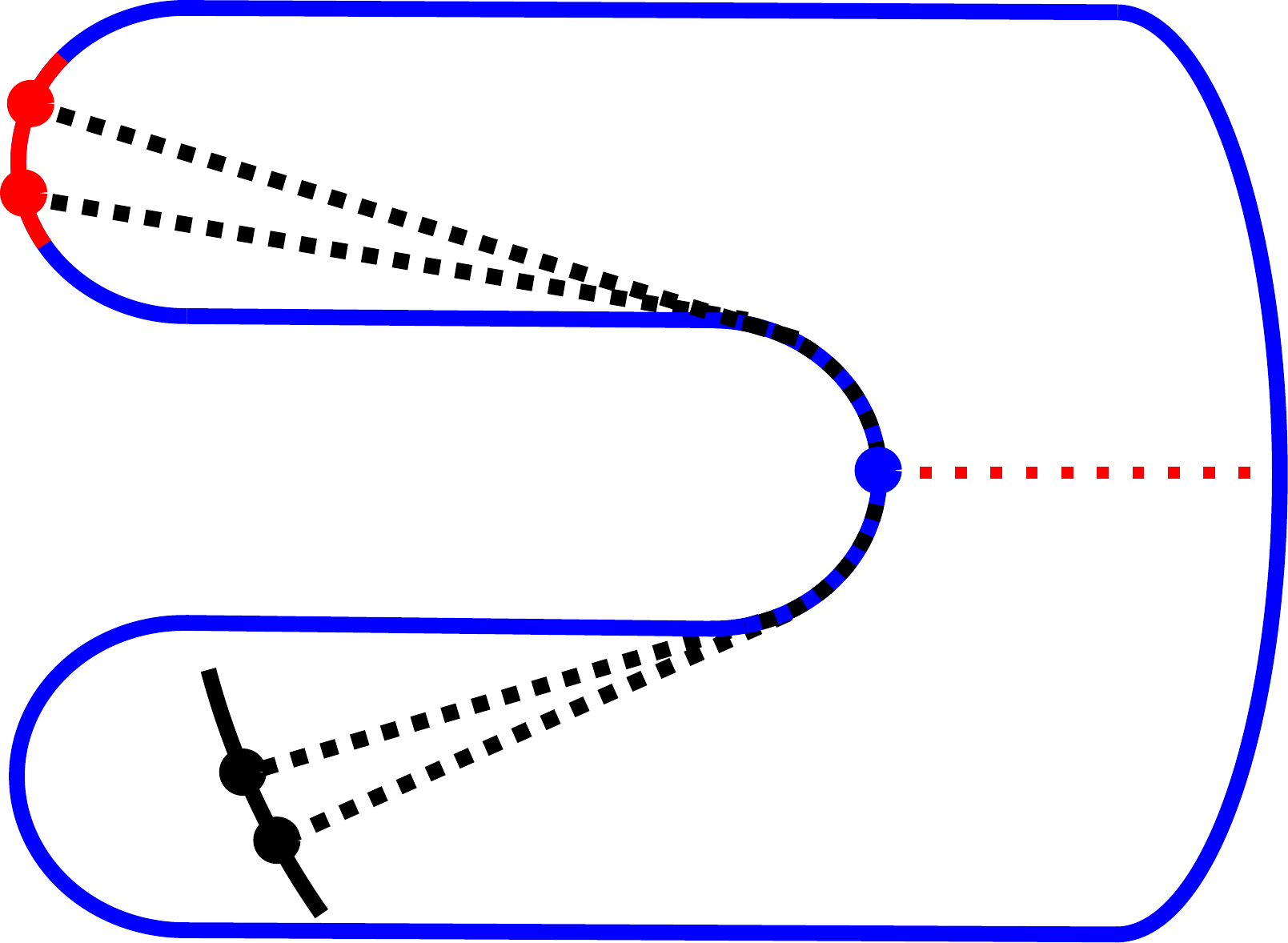}}
  \put(35,25){$P$}
   \put(175,15){$M_2$}
   \put(175,100){$x_0$}
    \put(-15,170){$\Gamma$}
    \put(175,175){$M_1$}
   \end{picture}
\caption{A domain where partial data is insufficient.
We split the domain~$M$ into two pieces~$M_1$ and~$M_2$ with respect to the line (red dotted line) that is normal to~$\p M$ at $x_0\in \p M$ (blue dot).
Then we choose a domain  $\Gamma \subset \p M_1$ (red arch) so that any minimizing curve joining a point on~$\Gamma$ and a point in~$M_2$ touches the boundary near~$x_0$.
The curve $P\subset M$ is any involute of the boundary, meaning that the distance from all points on~$P$ to~$x_0$ is the same.}
\label{Fi:f_p}
\end{figure}

Because~$P$ is an involute of the boundary as in the figure, $d(x,p)=d(x,q)$ for any $x\in\Gamma$ and $p,q\in P$.
Therefore from the point of view of our data, the set~$P$ appears to collapse to a point.

Worse, if~$\Gamma$ happens to be an involute as well, then all the distance functions to points $p\in M_2$ are constants in~$\Gamma$.
As the unknown origin times produce unknown constant offsets to the boundary distances, all points in~$M_2$ look alike when seen from an involutive~$\Gamma$ in this sense.

The problems of Figure~\ref{Fi:f_p} also illustrate the problems lack of convexity of the boundary may cause.

Similar problems may arise in higher dimensions if data is recorded on too small sets.
Let the manifold~$M$ be the closed unit ball in~$\R^n$ and consider partial data on $\Gamma=\partial M\cap H$, where $H\subset\R^n$ is a hyperplane through the origin.
Let~$p_1$ and~$p_2$ be two points in $\sisus(M)\setminus H$ situated symmetrically about~$H$.
Then the boundary distance functions of~$p_1$ and~$p_2$ agree on~$\Gamma$.
The sets~$M$ and $\Gamma\subset\partial M$ have a reflection symmetry which leaves the data invariant, making it impossible to distinguish the two sides.

\subsubsection{Infinitely many lentils}

The crucial condition for the estimates in Theorem~\ref{thm:C23} was that all of the lentils meet the spatial source set~$P$.
There is an infinite family of lentils, so we have a large number of conditions to check from our data.
The lentils are open and cover a certain compact subset of~$M$ (excluding a layer near the boundary), so in fact using a finite number will suffice.
However, it is not easy to identify a covering collection of lentils --- for which one would then easily check whether they contain source points --- from data.

\subsubsection{Conjugate points and boundary sources}
\label{sec:325}

If there is a source point on the boundary, the graph of the distance function is singular at that point.
The separation of the set~$Q(S)$ into graphs in Proposition~\ref{prop:separation} relies on smoothness.
If several corners happen to coincide, especially if $n=2$, it may be difficult to choose how to continue the graphs correctly.

This problem becomes far worse if there are conjugate points.
They also cause the boundary distance function to be non-smooth but can do so for several different points.
This makes both disentangling the data into graphs and the analysis of those graphs substantially more complicated.

The condition~\eqref{eq:diameter-curvature-bound} on the constants of Definition~\ref{def:bdd-geometry} is assumed in Theorems~\ref{thm:B}, \ref{thm:C23}, \ref{thm:reverse}, and~\ref{thm:D}.
This condition is used in Lemma~\ref{lma:lentil-diam} and Proposition~\ref{prop:transversal-radius} when the manifold is compared to model manifold with constant sectional curvature.
For the comparison to work, we need the model manifold of the same diameter to not have conjugate points, and this is exactly what the condition ensures.

The condition~\eqref{eq:diameter-curvature-bound} always holds in negative curvature.
If the manifold is known to have non-positive sectional curvature and an explicit bound on the diameter, one can simply choose the sectional curvature upper bound $\Cvi>0$ to be small enough to satisfy~\eqref{eq:diameter-curvature-bound}.
The condition also holds on all simple manifolds of constant sectional curvature.

The constants~$\Cxxviii$ of~\eqref{eq:c28},
$\Cxiii$ of~\eqref{eq:c13}, and
$\Cxxx$ of~\eqref{eq:c30},
as well as all constants derived from them,
become worse when~$\Civ\sqrt{\Cvi}$ gets close to~$\pi$.

\subsubsection{Constants of quantitative simplicity}

The constants of quantitative simplicity have to be the same for both manifolds~$M_i$ in Theorem~\ref{thm:C23}.
The set~$\Gamma_i$ and the approximate boundary inclusion map~$\alpha$ depend on these constants.
The distances between the source points are determined irrespective of the constants by Theorem~\ref{thm:A}, but the constants have an effect on the approximate boundary structure and the estimates on the quality of the reconstruction.

\subsubsection{Scaling of small quantities}

Consider the case when all the various epsilons are very small.
In the setting of Theorem~\ref{thm:C23} we have
$\delta\approx\eps_2\approx\eps^2$
and in Theorem~\ref{thm:reverse} we have
$\delta\approx\eps_2\lesssim\eps^2\approx\hat\eps$.
Ideally, we would expect to see no second powers so that all concepts of density --- the size of lentils, the true density, the observed quantity~$\eps_2$, the density estimated from data --- to be bi-Lipschitz equivalent to each other.
The difference in scaling is all due to the lentils being shaped so that
\begin{equation}
\text{outer radius}
\approx
\sqrt{\text{inner radius}}.
\end{equation}
This scaling is easy to verify in Euclidean geometry, and the relevant Riemannian aspects are covered in Lemmas~\ref{lma:lentil-diam} and~\ref{lma:trans-radius-and-ball}, and Proposition~\ref{prop:transversal-radius}.

\subsubsection{Generalization to simple Berwald manifolds}

As discussed above, our methods depend on the smoothness of the distance function. Thus these techniques are not applicable for general Riemannian metrics and relaxing the simplicity assumption would likely require different methods. There is however hope to apply these techniques to study analogous questions in the case of certain simple Finsler metrics.  

Any constant speed geodesic~$\gamma$ of a smooth connected and complete Finsler manifold $(N,F)$ is given by the solution of the geodesic equations, in local coordinates,
$
\ddot{\gamma}^k(t)+\Gamma^k_{ij}(\gamma,\dot{\gamma})\dot{\gamma}^i(t)\dot{\gamma}^j(t)=0.
$
Here $\Gamma^k_{ij}(x,v)$ are the coefficients of the \emph{Chern connection} on the tangent bundle~\cite{bao2012introduction}. Due to the dependence of the directional variable $v\in T_xN$ a reverse curve of a Finsler geodesic does not need to be a geodesic. If the Chern connection coefficients are directionally independent the Finsler metric~$F$ is called \emph{Berwald}. Thus Berwald manifolds have a well defined canonical covariant differential operator for tensor fields of any order. Therefore Berwald metrics carry a Gauss type of formula to connect boundary and interior Hessians for the distance function. This is not true for a general Finsler metrics and the lack of Gauss formula  would cause issues in several parts of our proofs. 
%
It is worth of mentioning that Berwald metrics are not just a theoretical curiosity but have a connection to linear elasticity. It has been observed that elastic Finsler metrics, arising from transverse isotropic medium in weak anisotropy, are actually Berwaldian~\cite{yajima2011finsler}. 

\section{Plan of proofs}
\label{sec:proof-plan}

Our data is given in the form of a set, a union of graphs, and we begin by disentangling it into the separate graphs.
We prove Proposition~\ref{prop:separation} to this effect in Section~\ref{sec:separation}.

In Section~\ref{sec:exact} we turn to finding distances between source points and other information that the data gives exactly.
With the help of Proposition~\ref{prop:separation}, we may start with the knowledge of the arrival time functions~$a_s$.
Each graph corresponds to a unique source point, giving us the first claim.
Computing suitable differences of arrival times, especially along the unique geodesic between two given source points, we cancel out unwanted contributions and find the distance between the two points.
Once we have these first two claims, the rest of Theorem~\ref{thm:A} follows straightforwardly.


In Section~\ref{sec:LGH} we depart from differential geometry to the geometry of metric spaces and study the labeled Gromov--Hausdorff distance in more detail, including the proof of Proposition~\ref{prop:LGH}.

The most substantial part of our proofs is in estimating the density of the sources, and we will do this in Section~\ref{sec:density} using the explicit bounds on geometric quantities found in Section~\ref{sec:bdd-geom} --- including the proof of Proposition~\ref{prop:bdd-geom}.

The first task is to estimate the density of sources near the boundary, which practically amounts to estimating how near the boundary a given source point is.
This is done in terms of the curvature of the graph of the arrival time function; when the point is very close to the boundary, the Hessian of the said function blows up.
We will prove that $d(p,x)\leq E(p,x)$ for any source point $p\in P$ and a critical point $x \in \p M$ of its arrival time function.
This gives an estimate in terms of data for how close to any given boundary point must there be a source point.
The initial data is first processed into the form of $E(p,x)$ and quantities derived from it, and we verify in Proposition~\ref{prop:E&L} that the necessary auxiliary quantities are indeed determined by our data.

To get a density estimate deep inside the manifold, a different set of tools must be employed.
A key tool is the lentil, an intersection of two metric balls.
Whether a source point belongs in a lentil defined by two other source points is entirely decided by the distances between the three points, and by Theorem~\ref{thm:A} this information is indeed determined by the data.
We will show that with a suitable choice of $\delta>0$ (as given in Theorem~\ref{thm:C23}) the lentils cover the deep interior of the manifold but are small.
This requires estimates in two directions for the lentils so that they are large enough to cover the relevant subset of~$M$ without gaps but are small enough so that they provide a useful density estimate.
If every point is contained in a lentil and that lentil contains a source point, then the density of sources is bounded by the diameters of the lentils.
The more complicated side of lentils is to ensure that they cover enough of the manifold, and this relies on a number of estimates based on bounded geometry.
These density estimates constitute a proof of Theorem~\ref{thm:C23}.

Density estimates turn out to be simpler in the reverse direction, showing that the methods used in the proof of Theorem~\ref{thm:C23} are fine enough to give a decent estimate of density.
These reverse estimates prove Theorem~\ref{thm:reverse}.

In Section~\ref{sec:convergence} we turn our attention to matters of convergence.
The proof of Theorem~\ref{thm:D} is mostly based on Theorem~\ref{thm:C23} and Theorem~\ref{thm:reverse}; as~$T$ increases, the source points become denser and denser both in reality and in terms of data-driven estimates, and thus the quality of the discrete approximation improves.

Theorem~\ref{thm:B} follows a similar path and can be regarded essentially as a corollary of Theorem~\ref{thm:D}, and thus it has a relatively simple proof now that all the tools are available.
The metric convergence results show that the two manifolds of Theorem~\ref{thm:B} are isometric as metric spaces, and by the Myers--Steenrod theorem the isometry is in fact smooth.
Proposition~\ref{prop:poisson} gets its proof, using similar ideas and basic properties of Poisson point processes, at the very end.

\section{Separation of graphs on a bundle}
\label{sec:separation}
Let us denote the scalar second fundamental form by~$h_2$ and the boundary metric (the first fundamental form) by~$h_1$. If $p \in \sisus(M)$, $\rho \colon M\to \R$ is the distance function $\rho(x)=d(p,x)$ at~$p$ then the corresponding boundary distance function $r\colon M \to \R$, is just the restriction of~$\rho$ on the boundary~$\p M$.
By the Gauss formula for hypersurfaces (see e.g. \cite[Theorem 8.13 a]{lee2018introduction}) and the definition of Hessians we have that
\begin{equation}
\label{eq:hess=hess+sff}
\Hess_r(x)
=
\Hess_\rho(x)
+
h_2(x).
\end{equation}
We will make use of this throughout the paper.

Our method of proving Proposition~\ref{prop:separation} is by means of lifting the data from the boundary $\partial M\times\R$ of the spacetime to a suitable bundle where the graphs do not intersect.
We start by showing that when $s\neq s'$, then the graphs of $a_s$ and $a_{s'}$ can only have up to first order tangency when they intersect.
Therefore $2$-jets will separate them.

\begin{lemma}
\label{lma:tangent-order}
Let~$M$ be a simple Riemannian manifold and $s,s'\in\sisus(M)\times\R$ be two distinct points in the spacetime.
Then if $a_s(x)=a_{s'}(x)$ at some point $x\in\partial M$, then the gradient or the Hessian of the two functions will differ at~$x$.
\end{lemma}

\begin{proof}
Take any two $s,s'\in S$ so that~$a_s$ and~$a_{s'}$, their differentials, and their Hessians agree at a point $x\in\p M$.
We aim to show that $s=s'$.

First, observe that the differential~$\der a_s(x)$ is the covector corresponding to the tangential component of the velocity of the unit speed geodesic from~$\pi(s)$ to~$x$.
Therefore $\der a_s(x)=\der a_{s'}(x)$ implies that $p\coloneqq\pi(s)$ and $p'\coloneqq\pi(s')$ are on the same geodesic starting from~$x$.

Let~$\gamma$ be this unit speed geodesic starting at~$x$.
We have $\gamma(t_0)=p$ and $\gamma(t_0')=p'$ for some $t_0,t_0'>0$.

Consider the Hessian $\Hess\coloneqq\Hess_{\rho_p}$ of the interior distance function $\rho_p\colon M\to\R$ defined by $\rho_s(x)=d(x,p)$ and its counterpart~$\Hess'$ corresponding to~$p'$.

If a normal Jacobi field~$J$ along~$\gamma$ satisfies $\Hess J(t)=-D_tJ(t)$ for some $t<t_0$, then $J(t_0)=0$, (see for instance \cite[Proposition 11.2]{lee2018introduction}).
Similarly, $\Hess'J(t)=-D_tJ(t)$ for any $t<t_0'$ implies $J(t_0')=0$.

Both Hessians satisfy $\Hess^{(\prime)}\dot\gamma(t)=-\dot\gamma(t)$ for $t<\min(t_0,t_0')$.
As the geodesic meets~$\p M$ transversally at~$x$ and the boundary Hessians of~$a_s$ and~$a_{s'}$ agree there, we have in fact the equality $\Hess(x)=\Hess'(x)$ for the Hessian operators on~$T_xM$.
This follows from~\eqref{eq:hess=hess+sff}.

Now take any nonzero $w\in T_xM$ normal to~$\dot\gamma(0)$ and let~$J$ be the Jacobi field along~$\gamma$ with $J(0)=w$ and $D_tJ(0)=\Hess w=\Hess'w$.
As observed above, this implies $J(t_0)=0$ and $J(t_0')=0$.
Due to the lack of conjugate points $p=\gamma(t_0)=\gamma(t_0')=p'$.

As the two source points are equal distance $t_0=t_0'$ from~$x$ and the two arrivals are at the same time, we also have $\tau(s)=\tau(s')$.
This concludes the proof of $s=s'$.
\end{proof}

\begin{proof}[Proof of Proposition~\ref{prop:separation}]
Let us abbreviate $Q\coloneqq Q(S)\subset\partial M\times\R$.
As $S\subset M\times\R$ is discrete and~$M$ is compact, there are only finitely many points of~$S$ in any interval $M\times[a,b]$.
By simplicity there is no geodesic longer than the diameter of the manifold.
Thus~$Q$ written as
\begin{equation}
Q
=
\bigcup_{s\in S}\Gr(a_s)
\end{equation}
is a locally finite union on $\partial M\times\R$.

Let~$\sigma(Q)$ be the set of ``smooth points of~$Q$'', where~$Q$ is given locally as a single graph.
By definition $\sigma(Q)\subset Q$ is open.
It is also dense, for otherwise there would be two graphs that coincide in an open set, contradicting Lemma~\ref{lma:tangent-order}.

Let~$E$ be the bundle over $\partial M\times\R$ with fibers
\begin{equation}
E_{(x,t)}
=
T_x\partial M\times T_x^{\otimes2}\partial M.
\end{equation}
This~$E$ is the product of the bundles~$T_xM$ and $T_xM\otimes T_xM$ pulled to $\partial M\times\R$ over the projection $\partial M\times\R\to\partial M$.

For any smooth function $h\colon U\to\R$ defined in an open set $U\subset\partial M$ we the lift its graph as $L\Gr(h)\subset E$ so that
\begin{equation}
L_{x,h(x)}\Gr(h)
=
(\nabla h(x),\nabla^2 h(x))
\in
T_x\partial M\times T_x^{\otimes 2}\partial M.
\end{equation}
This defines a smooth submanifold of~$E$.

In this way we lift all of~$\sigma(Q)$ into a submanifold~$L\sigma(Q)$ of~$E$.
Since $\sigma(Q)\subset Q$ is dense, we have
\begin{equation}
\label{eq:bundle-closure}
\overline{L\sigma(Q)}
=
\bigcup_{s\in S} L\Gr(a_s).
\end{equation}
There are no second order intersections by lemma~\ref{lma:tangent-order}, so the smooth submanifolds $L(\Gr(a_s))\subset E$ are pairwise disjoint.

Therefore~$Q$ determines the set
\begin{equation}
\{
L\Gr(a_s)\subset E
;\:
s\in S
\}
.
\end{equation}
Projecting from~$E$ down to $\partial M\times\R$ gives the graphs~$\Gr(a_s)$ and thus also the functions~$a_s$, proving the first part of the claim.

For the second claim we simply localize~\eqref{eq:bundle-closure} to be over $\Omega\subset\partial M\times\R$, and we get a disjoint union of the sets $L[\Gr(a_s)\cap \Omega]$.
A single graph~$\Gr(a_s)$ may be cut into several pieces by this procedure and we do not necessarily know which pieces correspond to the same source.
Therefore what we obtain is the collection of connected components of graphs restricted to~$\Omega$.
\end{proof}

It was proven in~\cite{kurylev1997multidimensional} that the boundary distance function $r_x\colon\partial M\to\R$ of $x\in\sisus(M)$ determines the point~$x$ uniquely.
The next lemma improves this slightly: the boundary distance function modulo constants is enough for this uniqueness.

\begin{lemma}
\label{lma:Slava+constant}
Let~$M$ be a simple Riemannian manifold and $x,y\in\sisus(M)$ any two distinct points.
Denote by~$r_x$ and~$r_y$ the boundary distance functions $\partial M\to\R$ from these points.
The difference function $r_x-r_y$ cannot be a constant function.
\end{lemma}

\begin{proof}
Suppose that the difference function $r_x-r_y$ takes the constant value $c\in\R$ despite $x\neq y$.
If $c=0$, the boundary distance functions coincide and thus by~\cite{kurylev1997multidimensional} $x=y$, yields a contradiction.
Therefore we may suppose that $c\neq0$.

Let~$\gamma$ be the maximal geodesic through the points~$x$ and~$y$ with endpoints $x',y'\in\partial M$ ordered so that~$x'$ is closer to~$x$ and~$y'$ to~$y$.
By assumption we have
\begin{equation}
d(x',x)
=
r_x(x')
=
r_y(x')+c
=
d(x',y)+c.
\end{equation}
Due to simplicity and the order of the points, we have $d(x',y)=d(x',x)+d(x,y)$.
Therefore $0=d(x,y)+c$.
The same calculation with the roles of~$x$ and~$y$ reversed shows that $0=d(x,y)-c$.
These two together imply that $d(x,y)=c=0$, which is a contradiction.
\end{proof}

\section{Exact observables}
\label{sec:exact}

\begin{proof}[Proof of Theorem~\ref{thm:A}]
Due to Proposition~\ref{prop:separation} the data determines the arrival time functions $a_s(x)=d(\pi(s),x)+\tau(s)$ at all points $x \in \p M$ for every $s\in S$.
It also follows from the same proposition that distinct sources have distinct arrival time functions.
Therefore now that the two manifolds~$M_1$ and~$M_2$ and their source sets~$S_1$ and~$S_2$ have the same data modulo identification by~$\phi$, we have a bijection between the source sets given by identifying the corresponding arrival time functions which coincide.
This proves claim~\ref{claim:i}.

We will prove all subsequent claims by describing how the arrival time functions determine the quantities in question.
As these functions coincide on the two manifolds, the reconstructed quantities coincide.
We therefore drop the subscripts and work with a single manifold.
We can think of the source points being indexed by their arrival time graphs.
To simplify notation, we denote the source points as $p_s=\pi(s)$ and source times as $t_s=\tau(s)$.

The arrival time functions now have the form $a_s(x)=d(p_s,x)+t_s$.
Two source points~$p_s$ and~$p_r$ coincide if and only if $a_s-a_r$ is a constant function --- this is a straightforward corollary of Lemma~\ref{lma:tangent-order} and Lemma~\ref{lma:Slava+constant}.
We pass to a subset of~$S$ so that each source point~$p_s$ is only present once; the claims extend easily to the duplicated source points.

For any $s,r\in S$ we define the functions $f_{rs}\colon\partial M\times\partial M\to\R$ by
\begin{equation}
f_{rs}(x,y)
=
a_r(x)-a_s(y).
\end{equation}
These functions are determined by the data. 
Taking $s=r$ proves claim~\ref{claim:v}.

From now on we suppose that $s\neq r$.
The differentials~$\der a_s$ and~$\der a_r$ agree at $x\in\partial M$ if and only if the geodesics joining~$x$ to~$p_s$ and~$p_r$ start in the same direction at~$x$.
Therefore the differentials agree at exactly two points, the endpoints of the unique maximal geodesic through the points~$p_s$ and~$p_r$.
Let these boundary points be~$x$ and~$y$.

Depending on how the four points ($p_s,p_r,x,y$) are ordered on the geodesic, we have
\begin{equation}
\label{eq:dist_between_sources}
\begin{split}
f_{rr}(x,y)-f_{ss}(x,y)
&=
d(p_r,x)-d(p_s,x)+d(p_s,y)-d(p_r,y)
\\&=
\pm 2d(p_r,p_s)
.
\end{split}
\end{equation}
Thus claim~\ref{claim:ii} is given by
\begin{equation}
d(p_i,p_j)
=
\frac{1}{2}
\abs{f_i(x,y)-f_j(x,y)}
.
\end{equation}
This determination of distances between source points will play a central role in the proofs of the other theorems.

By switching the points~$x$ and~$y$ if needed (depending on the sign in~\eqref{eq:dist_between_sources}), we can assume to have
$d(p_r,y)=d(p_r,p_s)+d(p_s,y)$.
Then for any $z\in\partial M$, we have
\begin{equation}
f_{rr}(z,y)-f_{ss}(z,y)+d(p_r,p_s)
=
d(p_r,z)-d(p_s,z).
\end{equation}
As everything on the left-hand side is determined by the data, so is then the right-hand side, and we have obtained the function $f^{rs}\colon\partial M\to\R$ given by
\begin{equation}
f^{rs}(z)
=
d(p_r,z)-d(p_s,z)
\end{equation}
for all $s,r\in S$.
This proves claim~\ref{claim:iv}.

Finally, we note that for any $z\in\partial M$ we have
\begin{equation}
t_r-t_s
=
f_{rs}(z,z)-f^{rs}(z)
\end{equation}
and claim~\ref{claim:iii} follows.
\end{proof}

As mentioned, the proof of Theorem~\ref{thm:B} will be postponed until we have approximation tools available.

\section{The labeled Gromov--Hausdorff distance}
\label{sec:LGH}

Our proof of the basic properties of the labeled Gromov--Hausdorff distance is similar to well known proofs of the usual Gromov--Hausdorff distance (see for instance  \cite[Chapter 3]{gromov2007metric} or \cite[Section 7.3]{burago2001course}) but with the labels interwoven into it.
For the sake of completeness we record our proof in its entirety, although the aspects unrelated to labels are indeed well known.

\begin{proof}[Proof of Proposition~\ref{prop:LGH}]
The symmetry of the labelled Gromov--Hausdorff distance is evident.
The triangle inequality can be verified as in \cite[Proposition 7.3.16]{burago2001course}.
Also, if there is~$h$ as described, then choosing $Z=Y$, $f=h$, and $g=\id$ shows that the distance is zero.
The only nontrivial claim is that if the distance is zero, then such an~$h$ exists.
Let us prove that. 

Take any compact metric space~$Z$ and isometric embeddings $f\colon X\to Z$ and $g\colon Y\to Z$.
We can define a semimetric (all properties of a metric but $d(x,y)=0$ need not imply $x=y$) $d$ on the disjoint union $W=X\sqcup Y$ by letting
\begin{equation}
d(a,b)
=
\begin{cases}
d_Z(f(a),f(b)),&\text{when }a\in X\text{ and }b\in X,\\
d_Z(f(a),g(b)),&\text{when }a\in X\text{ and }b\in Y,\\
d_Z(g(a),f(b)),&\text{when }a\in Y\text{ and }b\in X,\\
d_Z(g(a),g(b)),&\text{when }a\in Y\text{ and }b\in Y.
\end{cases}
\end{equation}
The embeddings are isometric, so $d_Z(f(a),f(b))=d_X(a,b)$ and similarly for~$Y$ and~$g$.

Take any $k\in\N$.
By assumption we have a metric space~$Z_k$ and thus a semimetric~$d_k$ on~$W$ so that $\hd{d_k}(X,Y)<\frac1k$ and $d_k(\alpha(\ell),\beta(\ell))<\frac1k$ for all $\ell\in L$.
These semimetrics agree with the metrics on~$X$ and~$Y$ when restricted to either set.

By compactness there is a finite set $A_k'\subset X$ so that the semiballs
\begin{equation}
\left\{
p\in X
;\:
d_k(x,p)<\frac1k
\right\}_{x\in A_k'}
\end{equation}
cover~$X$.
Similarly, there is a finite $A_k''\subset X$ so that the semiballs
\begin{equation}
\left\{
p\in Y
;\:
d_k(x,p)<\frac1k
\right\}_{x\in A_k''}
\end{equation}
cover~$Y$.
Similarly, there are finite sets~$B_k'$ and~$B_k''$ in~$Y$ so that the semiballs of radius~$\frac1k$ cover~$Y$ and~$X$. We define
$A_k=A_k'\cup A_k''$
and
$B_k=B_k'\cup B_k''$.
We can make all of these choices so that
$A_k\subset A_{k+1}$ and
$B_k\subset B_{k+1}$.

For any $k \in \N$ on the finite set $A_k\sqcup B_k$ the sequence of semimetrics $(d_j)_{j=k}^\infty$ is bounded pointwise, since $d_j|_{A_k \times A_k}=d_{j'}|_{A_k \times A_k}$ and $d_j|_{B_k \times B_k}=d_{j'}|_{B_k \times B_k}$ for $j,j' \geq k$. For $x\in A_k$ and $y \in B_k$ the assumption $\hd{d_k}(X,Y)<\frac1k$ yields
\begin{equation}
d_j(x,y)
\leq
\max(\diam(X),\diam(Y))+ \tfrac{2}{k}
\end{equation}
for all $j\geq k$.
Thus $(d_j)_{j=k}^\infty$ has a converging subsequence and the pointwise limit of semimetrics is a semimetric on $A_k\sqcup B_k$.
Constructing a diagonal sequence gives us a subsequence of~$(d_j)$ which converges pointwise on $W'=\bigcup_{k\in\N}(A_k\sqcup B_k)$ to a semimetric~$\delta$.
Because each~$d_k$ agrees with the original metrics on~$X$ and~$Y$, the only thing to inspect are the ``cross-distances'' between~$X$ and~$Y$.

We can extend~$\delta$ to all of~$W$ as follows.
When~$x$ and~$y$ are both in~$X$ or both in~$Y$, we use the metrics on these spaces.
When $x\in X$ and $y\in Y$, we pick for each $k\in\N$ points
$x_k\in B(x,\frac1k)\cap A_k\subset X$
and
$y_k\in B(y,\frac1k)\cap B_k\subset Y$.
We then let
\begin{equation}
\delta(x,y)
=
\lim_{k\to\infty}\delta(x_k,y_k).
\end{equation}
A simple argument shows that this limit is independent of the choice of the approximating sequences.
It is also straightforward to check that the extended~$\delta$ is indeed a semimetric on~$W$.

We want to show that for every $x\in X$ there is a unique $y\in Y$ so that $\delta(x,y)=0$.
The triangle inequality and~$\delta$ agreeing with the metric on~$Y$ shows that the point is unique.
For existence, there is a point in $x_k\in A_k$ that is $\frac1k$-close to~$x$, and there is $y_k\in B_k$ that is $\frac1k$-close to~$x_k$.
We can then set $y=\lim_{k\to\infty}y_k$.
Similarly, each $y\in Y$ has a unique $x\in X$ so that $\delta(x,y)=0$.

This gives rise to a bijection $h\colon X\to Y$ that satisfies $\delta(x,h(x))=0$.
This can be checked to be an isometry.

Take any $\ell\in L$ and $k\in\N$.
We have $d_k(\alpha(\ell),\beta(\ell))<\frac1k$ by construction of~$d_k$, so in the limit of the subsequence that we got we find $\delta(\alpha(\ell),\beta(\ell))=0$.
This means that $h(\alpha(\ell))=\beta(\ell)$.
Therefore $h\circ\alpha=\beta$.
\end{proof}

We record two additional propositions.
The proofs are immediate and we omit them.

\begin{proposition}
\label{prop:lgh-labels}
Let~$X$ and~$Y$ be compact metric spaces and~$L$ any set.
Let $\alpha\colon L\to X$ and $\beta\colon L\to Y$ be any two functions.
If $L'\subset L$, then
\begin{equation}
\ghd{L'}(X,\alpha|_{L'};Y,\beta|_{L'})
\leq
\ghd{L}(X,\alpha;Y,\beta).
\end{equation}
As $\ghd{\emptyset}(X,\emptyset;Y,\emptyset)$ is the usual Gromov--Hausdorff distance between~$X$ and~$Y$, the labeled kind of convergence implies the usual kind of convergence.
\end{proposition}

\begin{proposition}
\label{prop:lgh-density}
Let~$X$ be a compact metric space and $Y\subset X$ a subset.
Let $\alpha\colon L\to X$ and $\beta\colon L\to Y$ be any two functions on a set $L$.
If $Y\subset X$ is $\eps_1$-dense and $\sup_{\ell\in L}\abs{d(\alpha(\ell),\beta(\ell))}\leq\eps_2$,
then
\begin{equation}
\ghd{L}(X,\alpha;Y,\beta)
\leq
\eps_1+\eps_2.
\end{equation}
\end{proposition}

\section{Bounded geometry}
\label{sec:bdd-geom}

We begin our study of bounded geometry by proving Proposition~\ref{prop:bdd-geom} with the help of two lemmas, and then we move on to finding further estimates based on these basic bounds.

\begin{lemma}
\label{lma:global-logarithm}
Let~$M$ be a simple manifold, let $U_x\subset T_xM$ be the maximal domain of definition of the exponential map, and let $U\subset TM$ be the subset with fibers~$U_x$.
The map ${\exp}\times\pi\colon U\to M\times M$ that maps
\begin{equation}
U
\ni
(x,v)
\mapsto
(\exp_x(v),x)
\in
M\times M
\end{equation}
is a diffeomorphism and so has a smooth inverse $\theta\colon M\times M\to U\subset TM$.
\end{lemma}

\begin{proof}
As each $\exp_x\colon U_x\to M$ is a diffeomorphism on a simple manifold, the map ${\exp}\times\pi$ is clearly smooth and bijective.
What remains to check is the invertibility of the differential at every point $(x,v)\in U$.
The differential has a convenient block structure due to $\pi(x,v)=x$ being independent of~$v$, and so $\der({\exp}\times\pi)$ is invertible at $(x,v)$ if and only if $\der\exp_x$ is invertible at~$v$.
The invertibility of~$\der\exp_x$ is true by assumption.
\end{proof}

\begin{lemma}
\label{lma:Christoffel-bound}
Let $W^1,\dots,W^n$ be vector fields on a simple manifold~$M$ constituting a global orthonormal frame.
Let $\CS{i}{j}{k}(x,y)$ be the Christoffel symbol at $x\in M$ with the normal coordinates centered at $y\in M$ with $W^1(y),\dots,W^n(y)$ as the coordinate basis for~$T_yM$.
Then
\begin{equation}
\sup_{x,y\in M}
\abs{\CS{i}{j}{k}(x,y)}
<
\infty
\end{equation}
for all indices $i,j,k$.
\end{lemma}

There is a global orthonormal frame on every simple manifold given by vector fields~$W_i$ as stated.
Such a frame can be produced by the Gram--Schmidt method from a general frame coming from the trivializability of the tangent bundle.

\begin{proof}[Proof of Lemma~\ref{lma:Christoffel-bound}]
The vectors $W_1(y),\dots,W_n(y)$ constitute an orthonormal basis for~$T_yM$.
With the help of Lemma~\ref{lma:global-logarithm} we see that the basis vectors for~$T_xM$ in the normal coordinates of~$y$ are given by 
\begin{equation}
w_i(x,y)
\coloneqq
\der^{[x]}\exp_y(\theta(x,y))W_i(y).
\end{equation}
Here and later in this proof we indicate the variable of differentiation by superscript in [square brackets] when needed for clarity.

The invariant coordinate map for the normal coordinates of~$y$ is $\theta_y\colon M\to T_yM$ given by $\theta_y(x)=\theta(x,y)$.
Using the basis given by the frame, we get a proper coordinate map $\overline{\theta_y}\colon M\to\R^n$ given by $\overline{\theta_y}^i=W_i(y)^\flat\circ\theta_y$, which means
\begin{equation}
\overline{\theta_y}(x)
=
(
\ip{W_1(y)}{\theta_y(x)}
,
\dots
,
\ip{W_n(y)}{\theta_y(x)}
)
.
\end{equation}
Now we have a concrete description of the normal coordinates in terms of the frame.

Let us denote $w^i(x,y)=\der^{[x]}\overline{\theta_y}^i(x)\in T_x^*M$.
This is the differential of a coordinate map and thus a basis covector for~$T_x^*M$ induced by the normal coordinates of~$y$.
The corresponding basis vectors on~$T_xM$ are $w_i(x,y)$.
In terms of the coordinate map $z=\overline{\theta_y}\colon M\to\R^n$ the Christoffel symbols are given by
\begin{equation}
\begin{split}
\CS{i}{j}{k}(x,y)
&=
\der z^i(
\nabla_{\partial/\partial z^j}\frac{\partial}{\partial z^k}
)
\\&=
w^i(x,y)(
\nabla_{w_j(x,y)}^{[x]}w_k(x,y)
)
.
\end{split}
\end{equation}
All the items in this formula depend smoothly on~$x$ and~$y$, up to the boundary.
The claim then follows from compactness of $M\times M$.
\end{proof}

\begin{proof}[Proof of Proposition~\ref{prop:bdd-geom}]
Parts~\ref{bound:4}, \ref{bound:5}, \ref{bound:2}, and~\ref{bound:7} of Definition~\ref{def:bdd-geometry} follow easily from compactness.
We will prove the other parts more carefully.

Part~\ref{bound:6}:
Let us denote $y_i=\exp_x(\eta_i)\in M$.
In terms of Lemma~\ref{lma:Christoffel-bound} we have $\theta(y_i,x)=\eta_i$.
Let us equip~$TM$ with the Sasaki metric and $M\times M$ with the product metric.
An important property of the Sasaki metric on~$TM$ is that it respects the natural distance on the fibers: the Sasaki distance between $v,w\in T_xM$ is simply $\abs{v-w}$ because the fibers are totally geodesic~\cite{sasaki1958}.
Because~$\theta$ is a diffeomorphism between compact manifolds with boundary, it is Lipschitz-continuous.
The Lipschitz constant is denoted by~$\Cvii$, that is the maximum value of the norm of the Jacobian of~$\theta$.

Part~\ref{bound:1}:
Let us decompose an arbitrary Jacobi field~$J$ starting with $J(0)=0$ into parallel and normal components relative to~$\dot\gamma$:
\begin{equation}
J(t)
=
J^\parallel(t)
+
J^\perp(t).
\end{equation}
As $\abs{J(t)}^2=\abs{J^\parallel(t)}^2+\abs{J^\perp(t)}^2$ and both components are Jacobi fields, it suffices to prove the estimate separately for parallel and normal Jacobi fields.

If~$J$ is tangential to~$\gamma$ then $J(t)=tc\dot\gamma(t)$ for some $c\in\R$ and $\abs{J(t)}^2=c^2t^2$.
Thus for the tangential Jacobi fields the claim holds for any $\Ci\geq 2$.

Now we assume that~$J$ is normal to~$\dot{\gamma}$.
Let $\rho(x)=d(x,\gamma(0))$.
Then
\begin{equation}
\partial_t\abs{J(t)}^2
=
2\ip{D_t J}{J}
\leq
2\abs{D_t J}\abs{J}
=
2\abs{\Hess_\rho J}\abs{J}
\leq
2\abs{\lambda_{\mathrm{max}} (t)}\abs{J}^2,
\end{equation}
where $\lambda_{\mathrm{max}}(t)$ is the eigenvalue of the Hessian~$\Hess_\rho$ of~$\rho$ at~$\gamma(t)$ with the largest absolute value.

Let us choose normal coordinates at~$\gamma(0)$ as in Lemma~\ref{lma:Christoffel-bound}.
In these coordinates we have $\rho(x)=\abs{x}_e$ for all $x\in M$, where the subscript~$e$ stands for the Euclidean norm.
Geodesics with initial unit speed $v\in S_{\gamma(0)}M$ are given by $\gamma(t)=tv$.

Therefore the Hessian of the distance function~$\rho$ is given in these coordinates by
\begin{equation}
\label{eq:Hess-rho}
\Hess_\rho(x)
=
\frac{1}{|x|_e}
A_{ij}
\der x^i\otimes \der x^j,
\end{equation}
where
\begin{equation}
\label{eq:A-Hess-rho}
A_{ij}(x,\gamma(0))
=
\delta_{ij}-\frac{x_ix_j}{\abs{x}_e^2}-\CS{k}{i}{j}(x,\gamma(0))x_k.
\end{equation}
The numbers $A_{ij}$ are all bounded uniformly for all $\gamma(0)\in M$ and $x\in M\setminus\{\gamma(0)\}$ due to Lemma~\ref{lma:Christoffel-bound}.
Thus the Hessians~$\Hess_\rho$, as quadratic forms, are uniformly bounded by some constant~$\frac12 \Ci$, and
$\abs{\lambda_{\mathrm{max}} (t)}\leq\frac12\Ci$.
The claim follows.

Part~\ref{bound:8}:
This is similar to part~\ref{bound:1} and we continue with the same proof setup.
The point where everything is evaluated is $x=tv$, so by equations~\eqref{eq:Hess-rho} and~\eqref{eq:A-Hess-rho} we have
for all $w \in \dot\gamma(t)^\perp$
\begin{equation}
\ip{w}{\Hess_\rho w}
=
t^{-1}
[
\delta_{ij}w^iw^j
-
\CS{k}{i}{j}(tv)tv_kw^iw^j
]
.
\end{equation}
The bound on the Christoffel symbols implies that $\abs{\CS{k}{i}{j}(tv)}\leq C$ for some absolute constant $C\geq0$ independent of~$t$, $v$, and~$\gamma(0)$.
In the Euclidean norm $\abs{v^\flat}_e=1$, so $\abs{v_k}\leq1$.
Thus we find
\begin{equation}
\ip{w}{\Hess_\rho w}
\geq
t^{-1}\abs{w}_e^2
-
C\abs{w}_e^2
\end{equation}
for all $w \in \dot\gamma(t)^\perp$.
By the diffeomorphic nature of ${\exp}\times\pi$ of the proof of Lemma~\ref{lma:Christoffel-bound} the Euclidean norm of~$w$ is bi-Lipschitz equivalent with the Riemannian norm of~$w$, with bi-Lipschitz constants independent of~$x$, and~$\gamma(0)$.
This ensures that the desired two constants exist.
\end{proof}

We will next establish a number of estimates based on the results of Proposition~\ref{prop:bdd-geom}. Most of them we will rewrite in Section~\ref{sec:lentil} in the language of lentils.
Before getting started, we recall three standard comparison results, stated in a form suitable for our use.

We will only use Lemma~\ref{lma:rauch} to compare to constant sectional curvature, and we state the next two lemmas only in this setting.
We denote by~$M_\kappa$ the complete Riemannian manifold with constant sectional curvature $\kappa\in\R$, either the hyperbolic space, the Euclidean space, or the sphere, depending on the sign of~$\kappa$.

\begin{lemma}[{Rauch's comparison theorem \cite[Chapter 10, Theorem 2.3]{do1992riemannian}}]
\label{lma:rauch}
Let~$M$ be a simple Riemannian manifold with sectional curvature $\geq\kappa$.
Let~$J$ and~$\tilde J$ be Jacobi fields on unit speed geodesics~$\gamma$ and~$\tilde\gamma$ on the manifolds~$M$ and~$M_\kappa$, with $J(0)=0=\tilde J(0)$, $\ip{D_tJ(0)}{ \dot{\gamma}(0)}=\ip{D_t \tilde J(0)}{\dot{\tilde\gamma}(0)}$ and $\abs{D_tJ(0)}=\abs{D_t\tilde J(0)}$.
Then
\begin{equation}
\abs{J(t)}
\leq
\abs{\tilde J(t)}
\end{equation}
for all $t>0$ for which both geodesics are defined.
\end{lemma}


Consider a simple Riemannian manifold~$M$ with sectional curvature bounded from above by~$\Cvi$.
For any three points $x,y,z\in M$ we want to pick three points $\tilde x,\tilde y,\tilde z\in M_{\Cvi}$ so that
\begin{equation}
d(x,y)=d(\tilde x,\tilde y), \quad
d(x,z)=d(\tilde x,\tilde z), \quad
d(y,z)=d(\tilde y,\tilde z)
\end{equation}
or
\begin{equation}
d(x,y)=d(\tilde x,\tilde y), \quad
d(x,z)=d(\tilde x,\tilde z), \quad
\angle yxz=\angle \tilde y\tilde x\tilde z.
\end{equation}
We denote distances on both manifolds by~$d$, as there should be no confusion as to which space each point belongs.
As long as~$M$ satisfies the estimate~\eqref{eq:diameter-curvature-bound}, such points on~$M_{\Cvi}$ exist.

The following two lemmas compare triangles on~$M$ with corresponding triangles on~$M_\kappa$.
They both follow from Rauch's comparison theorem, although not exactly the same version as Lemma~\ref{lma:rauch} above.
We recall from Definition~\ref{def:bdd-geometry} that the sectional curvature is bounded from above by the constant $\Cvi>0$.

\begin{lemma}[{\cite[Chapter 10, Proposition 2.5]{do1992riemannian}}]
\label{lma:curvature-comparison2}
Let~$M$ be a simple manifold whose constants (of Definition~\ref{def:bdd-geometry} and Proposition~\ref{prop:bdd-geom}) satisfy the constraint~\eqref{eq:diameter-curvature-bound}.
Take any three points $x,y,z\in M$.
Then there are points $\tilde x,\tilde y,\tilde z\in M_{\Cvi}$ so that $d(x,y)=d(\tilde x,\tilde y)$,
$d(x,z)=d(\tilde x,\tilde z)$,
and the angle is the same at~$x$ and~$\tilde x$.
The remaining distances satisfy $d(y,z)\leq d(\tilde y,\tilde z)$.
\end{lemma}


Finally we give a comparison result for the angles of geodesic triangles. 

\begin{lemma}[{\cite[Theorem 2.7.6]{klingenberg}}]
\label{lma:curvature-comparison1}
Let~$M$ be a simple manifold whose constants (of Definition~\ref{def:bdd-geometry} and Proposition~\ref{prop:bdd-geom}) satisfy the constraint~\eqref{eq:diameter-curvature-bound}.
Take any three points $x,y,z\in M$.
Then there are points $\tilde x,\tilde y,\tilde z\in M_{\Cvi}$ so that
$d(x,y)=d(\tilde x,\tilde y)$,
$d(x,z)=d(\tilde x,\tilde z)$,
and
$d(y,z)=d(\tilde y,\tilde z)$,
and the angle at~$x$ is at most that at~$\tilde x$.
\end{lemma}


We will then proceed to estimates specific to our setting.

\begin{lemma}
\label{lma:JF-growth}
Let~$M$ be a simple Riemannian manifold of bounded geometry.
Let~$J$ be a Jacobi field along any constant (not necessarily unit)
speed geodesic with $J(0)=0$.
Then
\begin{equation}
\abs{J(t)}
\leq
\Cxiii
\abs{D_tJ(0)}
t
\end{equation}
for all $t>0$, where the constant is given by~\eqref{eq:c13}.
\end{lemma}

\begin{proof}
Let us first prove the claim for unit speed geodesics.

In constant sectional curvature $k=-\Cv<0$ every Jacobi field~$J$ with $J(0)=0$ and $\abs{D_tJ(0)}=1$ is of the form $J(t)=A(-k)^{-1/2}\sinh(t\sqrt{-k})$ with a parallel unit vector~$A$; see e.g. \cite[Proposition 10.12]{lee2018introduction}.

If~$M$ has sectional curvature $\geq k$, then by Lemma~\ref{lma:rauch} every Jacobi field~$J$ with $J(0)=0$ and $\abs{D_tJ(0)}=1$ satisfies $\abs{J(t)}\leq(-k)^{-1/2}\sinh(t\sqrt{-k})$ before a conjugate point.
As~$M$ is simple, there are no conjugate points.

The function $x\mapsto x^{-1}\sinh(x)$ is increasing for $x \in [0,\infty)$ and $\diam(M)\leq\Civ$.
Thus we get the desired estimate with the constant
\begin{equation}
\sup_{t\in (0,\Civ)}
t^{-1}(-k)^{-1/2}\sinh(t\sqrt{-k})
\leq 
\frac{\sinh(\Civ\sqrt{\Cv})}{\Civ\sqrt{\Cv}}
=
\Cxiii
\end{equation}
as claimed.

Consider then a geodesic~$\hat\gamma$ with non-unit constant speed $\lambda=\abs{\dot{\hat\gamma}(0)}>0$.
It can be written as $\hat\gamma(t)=\gamma(\lambda t)$ for a unit speed geodesic~$\gamma$.

Let $\Gamma(t,s)$ be a family of geodesics so that $\Gamma(t,0)=\gamma(t)$ and $J(t)=\partial_s\Gamma(t,s)|_{s=0}$ for a Jacobi field~$J$ along~$\gamma$ satisfying the assumptions of the claim.
Consider then the new Jacobi field
$\hat J(t')=\partial_{s'}\Gamma(\lambda t',\lambda^{-1}s')|_{s'=0}$ along the geodesic~$\hat\gamma$.
The scaling of the two parameters was chosen so that~$\hat J$ is a vector field along~$\hat\gamma$ and
$
D_{t'}\hat J(0)
=
D_tJ(0)
$,
so that both covariant derivatives have unit norm at $\gamma(0)=\hat\gamma(0)$.
As we have $\hat J(t')=\lambda^{-1}J(\lambda t')$, the already proven claim for~$J$ along~$\gamma$ gives
\begin{equation}
\begin{split}
\abs{\hat J(t')}
&=
\lambda^{-1}
\abs{\hat J(\lambda t')}
\\&\leq
\lambda^{-1}
\Cxiii
\lambda t'
\\&=
\Cxiii t'.
\end{split}
\end{equation}
This is the claimed estimate for~$\hat J$.
\end{proof}

\begin{lemma}
\label{lma:geodesic-diversion}
Let $\gamma_1,\gamma_2\colon[0,1]\to M$ be any two constant (not necessarily unit) speed geodesics with $\gamma_1(0)=\gamma_2(0)$.
Then
\begin{equation}
d(\gamma_1(t),\gamma_2(t))
\leq
\Cxiv
t
d(\gamma_1(1),\gamma_2(1))
\end{equation}
for all $t\in(0,1)$, where the constant is given by~\eqref{eq:c14}.
\end{lemma}

\begin{proof}
To do so, we write each geodesic as $\gamma_i(t)=\exp_x(t\eta_i)$.
Let~$\eta_s$ with $s\in[1,2]$ interpolate linearly between~$\eta_1$ and~$\eta_2$.

Then we may estimate the distance between the two required points by taking the $s$-curve with~$t$ constant.
First, the Jacobi field growth estimate of Lemma~\ref{lma:JF-growth} gives
\begin{equation}
\abs{\partial_s\exp_x(t\eta_s)}
\leq
\Cxiii
t
\abs{\eta_2-\eta_1}
.
\end{equation}
Therefore
\begin{equation}
\label{eq:vv68}
d(\gamma_1(t),\gamma_2(t))
\leq
\int_1^2\abs{\partial_s\exp_x(t\eta_s)}\der s
\leq
t\Cxiii\abs{\eta_2-\eta_1}
.
\end{equation}
By the uniform invertibility of the exponential maps due to bounded geometry we have
\begin{equation}
\abs{\eta_2-\eta_1}
\leq
\Cvii
d(\exp_x\eta_1,\exp_x\eta_2)
\end{equation}
and thus~\eqref{eq:vv68} turns into
\begin{equation}
d(\gamma_1(t),\gamma_2(t))
\leq
t\Cxiii\Cvii d(\gamma_1(1),\gamma_2(1)),
\end{equation}
as claimed.
\end{proof}

\section{Density estimates in bounded geometry}
\label{sec:density}

\subsection{Lentils}
\label{sec:lentil}

We remind the reader that a lentil is defined simply as
\begin{equation}
L^{x,y}_{r,s}
=
B(x,r)\cap B(y,s).
\end{equation}
We will always assume that $r,s<d(x,y)$.

We say that the thickness $\delta^{x,y}_{r,s}$ of this lentil is the length of the geodesic segment $\gamma_{x,y}\cap L^{x,y}_{r,s}$, where~$\gamma_{x,y}$ is the unit speed geodesic from~$x$ to~$y$.
Under the assumptions $r,s<d(x,y)$ this leads to
\begin{equation}
\label{eq:thickness}
\delta^{x,y}_{r,s}
=
r+s-d(x,y)
.
\end{equation}
The assumptions $r,s<d(x,y)$ imply that $\delta^{x,y}_{r,s}<r,s$.

The midpoint of the lentil is
\begin{equation}
m^{x,y}_{r,s}
=
\gamma_{x,y}(r-\tfrac12\delta).
\end{equation}
The transversal radius of the lentil~$L^{x,y}_{r,s}$ with midpoint $m=m^{x,y}_{r,s}$ is
\begin{equation}
R^{x,y}_{r,s}
=
\sup\{
R>0;\:
\exp_m(\eta)\in L
\text{ whenever }
\eta\in T_mM
\text{, }
\eta\perp\dot\gamma
\text{ and }
\abs{\eta}<R
\}.
\end{equation}
That is, (the closure of) the lentil contains a geodesic disc --- a ball of dimension $n-1$ --- of this radius normal to~$\gamma_{x,y}$.

\begin{lemma}
\label{lma:lentil-diam}
Consider a simple manifold of bounded geometry with the constants satisfying
the condition~\eqref{eq:diameter-curvature-bound}.
If a lentil $L^{x,y}_{r,s}$ has thickness $\delta>0$, then its diameter satisfies
\begin{equation}
\diam(L^{x,y}_{r,s})
\leq
\delta
+
\Cxv\sqrt{\delta},
\end{equation}
where the constant is given by~\eqref{eq:c15}.
\end{lemma}

The condition~\eqref{eq:diameter-curvature-bound} ensures that $\sin(\Civ\sqrt{\Cvi})>0$ and thus~$\Cxv$ is a well defined and positive constant.

\begin{proof}[Proof of Lemma~\ref{lma:lentil-diam}]
The proof is composed of two parts.
In the first part we show that if~$\delta$ is small, then the opening angle of the lentil as seen from~$x$ is small.
In the second part we use this to show that the lentil itself is small.

Let us denote $L\coloneqq L^{x,y}_{r,s}$.
Let $\gamma=\gamma_{x,y}$ be the geodesic between~$x$ and~$y$ and let $\ell\coloneqq d(x,y)$.
Take any point $z\in L$.

The three points $x,y,z$ form a triangle with sides $\ell$, $d(x,z)\eqqcolon a$, and $d(y,z)\eqqcolon b$.
Let us call the angle at~$x$ by the name $\beta\in[0,\pi]$.

We will estimate this angle by comparing to a constant curvature reference manifold.
As earlier, let~$M_{\Cvi}$ be the sphere of the same dimension as~$M$ with constant sectional curvature $\Cvi>0$.
In the corresponding geodesic triangle with the same side lengths $\ell,a,b$ the corresponding angle is~$\tilde\beta$.
By Lemma~\ref{lma:curvature-comparison1} we have $\beta\leq\tilde\beta$.

By the spherical law of cosines we have
\begin{equation}
\cos(\tilde\beta)
=
\frac{\cos(\tilde b)-\cos(\tilde\ell)\cos(\tilde a)}{\sin(\tilde\ell)\sin(\tilde a)},
\end{equation}
where we denoted
$\tilde\ell=\ell\sqrt{\Cvi}$, 
$\tilde a=a\sqrt{\Cvi}$, and
$\tilde b=b\sqrt{\Cvi}$.
We will also use scaled versions of the two radii and the thickness:
$\tilde r=r\sqrt{\Cvi}$,
$\tilde s=s\sqrt{\Cvi}$, and
$\tilde\delta=\delta\sqrt{\Cvi}$.

By the assumption that $z\in L$ we have $a<r$ and $b<s$, and the reverse triangle inequality gives $a>\ell-s$.
These same inequalities hold when scaled with~$\sqrt{\Cvi}$, whence
\begin{equation}
\cos(\tilde\beta)
\geq
\frac{\cos(\tilde s)-\cos(\tilde\ell)\cos(\tilde\ell-\tilde s)}{\sin(\tilde\ell)\sin(\tilde r)}.
\end{equation}
Using
\begin{equation}
\cos(\tilde s)
=
\cos(\tilde\ell)\cos(\tilde\ell-\tilde s)
+
\sin(\tilde\ell)\sin(\tilde\ell-\tilde s)
\end{equation}
and $\ell-s=r-\delta$,
the estimate simplifies to
\begin{equation}
\cos(\tilde\beta)
\geq
\frac{\sin(\tilde r-\tilde\delta)}{\sin(\tilde r)}.
\end{equation}
We will simplify this further using
$\sin(\tilde r-\tilde\delta)\geq\sin(\tilde r)-\tilde\delta$.

Finally, we make use of the fact that $\sin(t)\geq T^{-1}\sin(T)t$ for all $t\in[0,T]$ when $T\in(0,\pi)$.
All the distances on the manifold are bounded by the diameter bound~$\Civ$, so in our setting $T=\Civ\sqrt{\Cvi}$, which was indeed assumed to be below~$\pi$.
Combining these, we find
\begin{equation}
\cos(\tilde\beta)
\geq
1
-
\Cxxii\frac{\delta}{r}
\end{equation}
with the constant given by~\eqref{eq:c22}.
(In the limit $\Cvi\to0$ we have $\Cxxii=1$, and one can indeed verify the estimate with this constant using Euclidean comparison geometry when the sectional curvature is non-positive.)

Using $\cos(\beta)\leq1-\frac15\beta^2$, which is valid for all $\beta\in[0,\pi]$, we get
\begin{equation}
1-\tfrac15\beta^2
\geq
\cos(\beta)
\geq
1-\Cxxii\frac{\delta}{r},
\end{equation}
which leads to
\begin{equation}
\label{eq:beta^2<d/r}
\beta^2
\leq
5\Cxxii
\frac{\delta}{r}.
\end{equation}
This explicit estimate for~$\beta$ in terms of~$\delta$ and~$r$ concludes the first part of the proof.

Now consider any two points $z_1,z_2\in L$.
Estimating their distance will amount to estimating the diameter of the lentil.
Let~$\beta_i$ be the angle at~$x$ between~$\gamma_{x,y}$ and~$\gamma_{x,z_i}$.

Let~$\hat z_1$ be the unique point on~$\gamma$ with $d(x,z_1)=d(x,\hat z_1)$.
The point~$\hat z_1$ is the closest one to~$y$ on the metric sphere $S(x,d(x,z_1))$ because~$\gamma$ meets the sphere orthogonally.
Thus $d(y,\hat z_1)\leq d(y,z_1)<s$.
As also $d(x,\hat z_1)=d(x,z_1)<r$, we have $\hat z_1\in L$.
Let $\eta,\hat\eta\in T_xM$ be the vectors for which
$\exp_x(\eta)=z_1$
and
$\exp_x(\hat\eta)=\hat z_1$.
Let $\tilde\sigma\colon[0,\beta_1]\to T_xM$ be the unit speed circular arc with constant norm and endpoints $\tilde\sigma(0)=\hat\eta$ and $\tilde\sigma(\beta_1)=\hat\eta_1$.
Then $\sigma={\exp_x}\circ\tilde\sigma$ is a curve joining~$\hat z_1$ to~$z_1$.

If sectional curvature and length of geodesics is bounded uniformly, then
\begin{equation}
\abs{\dot\sigma(t)}
\leq
\Cxiii d(x,z_1)
<
\Cxiii r
\end{equation}
by lemma~\ref{lma:JF-growth}.
Estimating the distance between the two points by the length of a curve joining them, we thus have $d(z_1,\hat z_1)\leq \Cxiii r\beta_1$.

By~\eqref{eq:beta^2<d/r} we have $\beta_i\leq\sqrt{5\Cxxii\delta/r}$, and clearly $d(\hat z_1,\hat z_2)<\delta$ when~$\hat z_2$ is defined analogously to~$\hat z_1$ above.
The triangle inequality gives
\begin{equation}
d(z_1,z_2)
\leq
\Cxiii r\beta_1
+
\delta
+
\Cxiii r\beta_2
\leq
\delta
+
2\sqrt{5\Cxxii}\Cxiii\sqrt{r\delta}.
\end{equation}
Thus the diameter of the lentil can be estimated by
\begin{equation}
\diam(L)
\leq
\delta
+
2\sqrt{5\Cxxii}\Cxiii\sqrt{\diam(M)}\sqrt{\delta}
\end{equation}
as claimed.
\end{proof}

\begin{lemma}
\label{lma:trans-radius-and-ball}
Consider a simple manifold~$M$ of bounded geometry with the constants satisfying the condition~\eqref{eq:diameter-curvature-bound}.
Take any $x,y\in M$ and consider the lentil $L^{x,y}_{r,s}$ with $r,s\in(0,d(x,y))$.
If~$m$ is the midpoint of~$L^{x,y}_{r,s}$, then $B(m,\frac12\delta)\subset L^{x,y}_{r,s}$ and thus $R^{x,y}_{r,s}\geq\frac12\delta$.
\end{lemma}

\begin{proof}
Recall that $L^{x,y}_{r,s}=B(x,r)\cap B(y,s)$.
As the distances from~$m$ to~$x$ and~$y$ are $r-\frac12\delta$ and $s-\frac12\delta$, respectively, it follows from the triangle inequality that $B(m,\frac12\delta)$ is contained in both balls defining the lentil.
The estimate for the transversal radius follows immediately.
\end{proof}

Lemma~\ref{lma:trans-radius-and-ball} is sufficient for our needs, but we point out that the exponent of~$\delta$ can be improved.
We provide this observation as Proposition~\ref{prop:transversal-radius} in Appendix~\ref{app:transversal}.

The endpoints~$x$ and~$y$ of the relevant geodesics will be source points in $\Gamma\subset B(\partial M,\eps_1)$.
As the source set is countable in practice and can well be finite, the geodesics will not cover the deep interior of the manifold.
But if the geodesics are fattened suitably, with the suitable level given by lemma~\ref{lma:dense-geodesics}, they do cover the desired set without gaps.
Such a thickening is most naturally realized in our setting as a union of lentils with~$x$, $y$, and $\delta>0$ fixed.
Lemma~\ref{lma:trans-radius-and-ball} and Proposition~\ref{prop:transversal-radius} gives a lower bound on the fatness of this ``sausage''.

\subsection{Boundary density estimates}

We will now describe how the quantities defined in equations~\eqref{eq:E1}, \eqref{eq:E2}, and~\eqref{eq:E3} estimate the distance of source points to the boundary.

Consider a source $s\in S$ located at the point $p=\pi(s)$.
The boundary distance function $r_p\colon\partial M\to\R$ is given by $r_p(x)=d(p,x)$ and the arrival time function~$a_s$ is this function shifted by a constant.
Proximity of a source point $p\in\sisus(M)$ to the boundary causes the Hessians of both these functions to blow up.
However, some of the curvature of the graph can be caused by the curvature of the boundary itself rather than the proximity of the source point.
Bounded geometry gives a concrete bound on this effect and allows us to ensure with concrete estimates that the source point really is close to the boundary.

Recall that $c(p)\subset\partial M$ is the set of critical points of the boundary distance function~$r_p$.

\begin{lemma}
\label{lma:d<E}
Let~$M$ be a simple Riemannian manifold with bounded geometry and $\emptyset\neq P\subset\sisus(M)$.
Let $\eps_1>0$.
The quantities defined in equations~\eqref{eq:E1}, \eqref{eq:E2}, and~\eqref{eq:E3} satisfy the following estimates:
\begin{enumerate}
\item
\label{item:E1}
Take any $p\in P$.
If $x\in c(p)$, then $d(p,x)\leq E(p,x)$.
\item
\label{item:E2}
For any $x\in\p M$ there is $p\in P$ with $d(p,x)\leq E(x)+\eps_1$.
\item
\label{item:E3}
For any $x\in\partial M$ there is $p\in P$ so that $d(p,x)\leq E+\eps_1$.
\end{enumerate}
\end{lemma}

\begin{proof}
Part~\ref{item:E1}:
Take any $p\in\sisus(M)$.
Let $\rho\colon M\to\R$ be the distance function $\rho(x)=d(p,x)$ and $r=\rho|_{\partial M}$.
The Hessians~$\Hess_\rho$ and~$\Hess_r$ of these functions are symmetric quadratic forms on~$TM$ and~$T\p M$, respectively.

Let~$\gamma$ be any unit speed geodesic with $\gamma(0)=p$ and~$J$ a Jacobi field along~$\gamma$ so that $J(0)=0$.
Suppose~$\gamma$ meets the boundary orthogonally, which is equivalent with the exit point being a critical point of~$r$.
Then the Hessian of the distance function has the property (cf. proof of Lemma~\ref{lma:tangent-order})
\begin{equation}
D_tJ(t)
=
\Hess_\rho J(t)
\end{equation}
for all $t>0$ and thus by part~\ref{bound:1} of Definition~\ref{def:bdd-geometry} we also have
\begin{equation}
\ip{J(t)}{\Hess_\rho J(t)}
\leq
\Ci t^{-1}\abs{J(t)}^2.
\end{equation}
As there are no conjugate points, this amounts to
\begin{equation}
\Hess_\rho
\leq
\Ci t^{-1}g
\end{equation}
in the sense of quadratic forms on $\dot\gamma^\perp\subset T_{\gamma(t)}M$.

If~$\lambda(x)$ denotes the smallest eigenvalue of~$\Hess_r(x)$ at a boundary point $x\in\partial M$, then~\eqref{eq:hess=hess+sff} yields
\begin{equation}
\label{eq:Hess-estimate-chain}
\lambda(x)
h_1(x)
\leq
\Hess_r(x)
=
\Hess_\rho(x)
+
h_2(x)
\leq
[\Ci r(x)^{-1}+\Cii]
h_1(x)
\end{equation}
as quadratic forms on the boundary.
If $\lambda(x)>\Cii$, this estimate gives
\begin{equation}
r(x)
\leq
\frac{\Ci}{\lambda(x)-\Cii}
=:
E(p,x).
\end{equation}
If $\lambda(x)\leq \Cii$, then $E(p,x)=\infty$.
The claimed estimate thus holds in both cases.

Part~\ref{item:E2}:
Let $x \in \p M$.
By the triangle inequality,
the estimate $d(y,x)\leq d_{\p M}(y,x)$ and Part~\ref{item:E1} we have
\begin{equation}
d(p,x)
\leq
d(p,y)+d_{\p M}(y,x)
\leq
E(p,y)+d_{\p M}(y,x)
\end{equation}
for all $p \in P$ and $y \in c(p)$.
Thus by the definition of $E(x)$ in~\eqref{eq:E2} there is a point $p\in P$ with a distance to~$x$ less than $E(x)+\eps_1$ to~$x$.

Part~\ref{item:E3}:
Follows immediately from the previous one.
\end{proof}

With the aid of Lemma~\ref{lma:d<E} we can now prove Proposition~\ref{prop:E&L}.

\begin{proof}[Proof of Proposition~\ref{prop:E&L}]
Proposition~\ref{prop:separation} implies that the data determines the arrival time functions~$a_s$ for all $s \in S$, although we do not have a description of the index set~$S$ yet.
As these functions differ from the boundary distance functions $r_s\colon\p M\to\R$, $r_s(x)=d(x,\pi(s))$, only by a constant, we thus know the differential and the Hessian of each~$r_s$ on all of~$\p M$.
Therefore the data determines the critical points of these functions and the function $E(p,y)$, for $p=\pi(s)$ and $y \in c(p)$, of~\eqref{eq:E1}.
From this one can easily compute~$E(x)$ and~$E$ from~\eqref{eq:E2} and~\eqref{eq:E3}.

The set~$\Gamma$ is defined in terms of these quantities, so it is uniquely determined by the data as well.

Theorem~\ref{thm:A} indicates that the data determines the pointwise spatial distances of the source points, and the last part of the claim follows.
\end{proof}

\subsection{Interior density estimates}

For density of sources deep in the manifold~$M$, we use two estimates.
The first one concerns the density of the set of geodesics connecting near-boundary source points in~$\Gamma$.
The second one ensures that the lentils cover enough of the manifold.

\begin{lemma}
\label{lma:dense-geodesics}
Let~$M$ be a simple manifold of bounded geometry and let $\eps_2>0$.
Let $\Gamma\subset M$ be such that $\partial M\subset B(\Gamma,\eps_2)$.
Then for every $z\in\sisus(M)$ there are points $x,y\in\Gamma$ so that
\begin{equation}
\dist(z,\gamma_{x,y}([0,d(x,y)]))
\leq
\Cxiv\eps_2.
\end{equation}
\end{lemma}

\begin{proof}
Fix any $x\in\Gamma$ and take any $z\in\sisus(M)$.
Let $\hat\gamma=\gamma_{x,z}$ be a constant speed geodesic with $\hat\gamma(0)=x$ and $\hat\gamma(t)=z$ for some $t>0$.
We extend this geodesic beyond~$z$ so that it meets~$\partial M$ at some time $t'>t$.
We scale the constant speed so that $t'=1$ and we denote  $\hat y\coloneqq\hat\gamma(1)\in\p M$.

There is $y\in\Gamma$ so that $d(y,\hat y)<\eps_2$.
Let $\gamma\colon[0,1]\to M$ be the constant speed geodesic for which $\gamma(0)=x$ and $\gamma(1)=y$.

By Lemma~\ref{lma:geodesic-diversion} we have
\begin{equation}
d(\hat\gamma(t),\gamma(t))
\leq
\Cxiv td(\hat y,y).
\end{equation}
As $t<1$, we have thus
\begin{equation}
\dist(z,\gamma([0,1]))
\leq
d(\hat\gamma(t),\gamma(t))
<
\Cxiv\eps_2
\end{equation}
as claimed.
\end{proof}

\begin{lemma}
\label{lma:lentils-cover}
Take any $\eps_1>0$ and $\eps_2>0$,
and 
let~$M$ be a simple manifold with bounded geometry and $\Gamma\subset B(\partial M,\eps_1)$
so that $\partial M\subset B(\Gamma,\eps_2)$.
Let $z \in M$. Whenever $d(z,\partial M)\geq\eps_1+\Cxvii\eps_2$,
then there is a lentil~$L^{x,y}_{r,s}$ with $x,y \in \Gamma$ and thickness $\delta=\Cix\eps_2<\min(r,s)$ containing~$z$.
The constants are given by~\eqref{eq:c9} and~\eqref{eq:c17}.
\end{lemma}

\begin{proof}
We showed in Lemma~\ref{lma:dense-geodesics} above that for any $z\in\sisus(M)$ there are $x,y\in\Gamma$ so that the distance from~$z$ to the trace of the geodesic~$\gamma_{x,y}$ is at most $\Cxiv\eps_2$.
We will choose $h>0$ later (see~\eqref{eq:h-choice} for an explicit expression) and require that $d(x,z)\geq h$ and $d(y,z)\geq h$.
These additional requirements
may not hold for the pair of points $x,y\in\Gamma$ without further assumptions for~$z$.
As both~$x$ and~$y$ are $\eps_1$-close to the boundary, the additional requirements $d(x,z)\geq h$ and $d(y,z)\geq h$ are certainly satisfied if $d(z,\partial M)\geq \eps_1+h$.
This is why we placed a boundary distance assumption on~$z$ in the statement of this lemma.

Let $z'\in\gamma_{x,y}$ be the nearest point on this geodesic to the point~$z$. After choosing~$h$ large enough we obtain by the reverse triangle inequality
\begin{equation}
\label{eq:rev-triangle-lentil-cover}
d(x,z'), d(y,z')
\geq
h
-
\Cxiv\eps_2
>
0
.
\end{equation}
This implies that~$z'$ is an interior point of~$\gamma_{x,y}$, and the geodesics~$\gamma_{x,y}$ and~$\gamma_{z,z'}$ meet orthogonally at~$z'$.
To make sure that~$z$ is contained in a lentil~$L^{x,y}_{r,s}$, with thickness~$\delta$ and midpoint~$z'$, we require that the transversal radius of that lentil is more than $d(z,z')$.

By Lemma~\ref{lma:trans-radius-and-ball} the aforementioned condition of the transversal radius is satisfied when
\begin{equation}
\label{eq:cond-1}
\Cxiv\eps_2
\leq
\tfrac12\delta
.
\end{equation}
The condition
for~$z'$ to be a midpoint of a lentil of thickness~$\delta$ is that $d(x,z')>\frac12\delta$ and $d(y,z')>\frac12\delta$, which by equation~\eqref{eq:rev-triangle-lentil-cover} are satisfied when
\begin{equation}
\label{eq:cond-2}
h-\Cxiv\eps_2
>
\tfrac12\delta
.
\end{equation}
Combining the two conditions~\eqref{eq:cond-1} and~\eqref{eq:cond-2} gives
\begin{equation}
\Cxiv\eps_2
<
h-\Cxiv\eps_2
\end{equation}
and therefore
\begin{equation}
h
>
2\Cxiv
\eps_2.
\end{equation}
We choose
\begin{equation}
\label{eq:h-choice}
h
=
3\Cxiv
\eps_2.
\end{equation}
The conditions~\eqref{eq:cond-1} and~\eqref{eq:cond-2} then become
\begin{equation}
2\Cxiv\eps_2
\leq
\delta
<
4\Cxiv\eps_2,
\end{equation}
so we may choose
\begin{equation}
\label{eq:c14c16_inv}
\delta
=
2\Cxiv
\eps_2.
\end{equation}
This is why we chose~$\Cix$ in~\eqref{eq:c9} as we did.

Finally we note that by the choice of $z'\in \gamma_{x,y}$ and due to~\eqref{eq:rev-triangle-lentil-cover}, \eqref{eq:h-choice}, and~\eqref{eq:c14c16_inv} we have
$
d(x,y)>
\delta.
$
If we set $r=d(x,z')+\tfrac12\delta$ and $s=\delta+d(x,y)-r=d(z',y)+\tfrac12\delta$,
then the lentil~$L_{r,s}^{x,y}$ is of width~$\delta$ and has~$z'$ as a midpoint. Moreover, by the previous argument this lentil contains the point~$z$, if $d(z,\partial M)\geq\eps_1+\Cxvii\eps_2$, as required in the claim of this lemma. 
\end{proof}

The choices of~$h$ and~$\delta$ in the proof above are optimal up to a constant and the explicit choice simplifies our estimates.

\subsection{Global density estimates}
\label{sec:thm-C-pf}

With all the estimates we have collected, we are ready to prove both parts of Theorem~\ref{thm:C23}.
Claim~\ref{claim:2} states that if all lentils of thickness~$\delta$ meet a source point (which is verifiable from data), then the source set $P_1\subset M_1$ has a concrete density estimate.
The thickness~$\delta$ is chosen carefully depending on~$\eps_2$.
Claim~\ref{claim:3} states that under the assumptions of the previous claim the measurements define a discrete metric space which is a quantitatively good approximation for the true space.

\begin{proof}[Proof of Theorem~\ref{thm:C23}] 
For claim~\ref{claim:2}, take any $z\in M$.
We wish to show that there is $p\in P$ so that $d(z,p)<\eps$ with~$\eps$ chosen as in~\eqref{eq:eps<eps2+reps2}.
We split the proof in two cases: near the boundary and deep within the manifold, and use different tools in either of these cases to verify the validity of the density claim. 

The boundary case:
If $\dist(z,\p M)<\eps_1+\Cxvii\eps_2$,
then there is $x\in\p M$ so that $d(z,x)<\eps_1+\Cxvii\eps_2$.
By claim~\ref{item:E3} of Lemma~\ref{lma:d<E}, there is $p\in P$ so that $d(x,p)<E+\eps_1$.
Recall that $\eps_2=E+\eps_1>\eps_1$.
By the triangle inequality we have
\begin{equation}
\label{eq:z-bdy-d-P}
\dist(z,P)
\leq
d(z,p)
\leq
d(z,x)
+
d(x,p)
<
(\Cxvii+2)\eps_2.
\end{equation}
This is a sufficient estimate near the boundary.

The deep interior case:
By definition of the set~$\Gamma$ in~\eqref{eq:Gamma} and Lemma~\ref{lma:d<E} we have $\Gamma\subset B(\p M, \eps_1)$ and $\partial M\subset B(\Gamma,\eps_2)$.
If $d(z,\p M)\geq\eps_1+\Cxvii\eps_2$, 
then Lemma~\ref{lma:lentils-cover} shows that there is a lentil $L=L^{x,y}_{r,s}$ of the correct thickness $\delta=\Cix\eps_2<r,s$ so that $z\in L$. 
By the crucial assumption of the theorem, there is $p\in P\cap L$.
Lemma~\ref{lma:lentil-diam} then gives us
\begin{equation}
\label{eq:z-int-d-P}
d(z,P)
\leq
d(z,p)
\leq
\delta
+
\Cxv\sqrt{\delta}
=
\Cix\eps_2
+
\Cxv\sqrt{\Cix\eps_2}.
\end{equation}
This is a sufficient estimate for deep interior points.

Now it remains to combine the estimates for points in the deep interior and near the boundary.
Combining~\eqref{eq:z-bdy-d-P} and~\eqref{eq:z-int-d-P} gives
\begin{equation}
\label{eq:pf-density-estimate}
\begin{split}
d(z,P)
&\leq
\max\left(
(\Cxvii+2)\eps_2
,
\Cix\eps_2
+
\Cxv\sqrt{\Cix\eps_2}
\right)
\\&\leq
(\Cxvii+2+\Cix)\eps_2  
+
\Cxv\sqrt{\Cix\eps_2}
\end{split}
\end{equation}
for all $z\in M$. 
This proves that $P\subset M$ is $\eps$-dense with the choice of~\eqref{eq:eps<eps2+reps2} and concludes the proof of claim~\ref{claim:2}.

The labeled Gromov--Hausdorff distance is now straightforward to estimate for claim~\ref{claim:3}.
By the definition of~$E(x)$ in~\eqref{eq:E2} we can choose a function $\alpha \colon \p M \to P$ that satisfies~\eqref{eq:alpha-choice}.
This function can be actually identified from the data~$Q(S)$ as we know the first fundamental form of~$\p M$ and the value of~$E(p,y)$ for all $p \in \pi(S)$ and $y \in c(p)$, given by~\eqref{eq:E1}, is obtained by the computing the gradient and the Hessian of the corresponding arrival time function.
Then by Lemma~\ref{lma:d<E} it follows that for each $x \in \p M$ there exists $y\in c(\alpha(x))$ that satisfies
\begin{equation}
\eps_2\geq E(x)+\eps_1
>
E(\alpha(x),y)+d_{\p M}(x,y)
\geq
d(\alpha(x),y)+d(x,y)
\geq
d(\alpha(x),x)
.
\end{equation}
Since $P\subset M$ is $\eps$-dense, Proposition~\ref{prop:lgh-density} finally gives
\begin{equation}
\ghd{\p M}(P,\alpha;M,\iota)
\leq
\eps
+
\eps_2.
\end{equation}
By the choice of~$\eps$ in~\eqref{eq:eps<eps2+reps2} this proves estimate~\eqref{eq:ghd<eps2+reps2} and completes the proof.
\end{proof}

\subsection{Reverse density estimates}

We do not need any more preparations before the next proof.

\begin{proof}[Proof of Theorem~\ref{thm:reverse}]
Consider a unit speed geodesic~$\gamma$ starting at a source point $p=\gamma(0)\in P$.
Take any $t>0$ and any $w\in T_{\gamma(t)}M$ orthogonal to~$\dot\gamma(t)$.
By bounded geometry we have
\begin{equation}
\ip{w}{\Hess_\rho w}
\geq
\abs{w}^2
(\Cxx t^{-1}-\Cxxi).
\end{equation}
Now suppose that $y=\gamma(t)$ is on~$\p M$ and~$\dot\gamma(t)$ is normal to the boundary.

The second fundamental form~$h_2$ is positive definite by simplicity and by~\eqref{eq:hess=hess+sff} we deduce
\begin{equation}
\Hess_r(y)
=
\Hess_\rho(y)+h_2(y)
\geq
(\Cxx t^{-1}-\Cxxi)h_1(y),
\end{equation}
where~$h_1$ is the first fundamental form.
Therefore the smallest eigenvalue~$\lambda(p,y)$ of~$\Hess_r(y)$ satisfies
\begin{equation}
\label{eq:lambda>1/d}
\lambda(p,y)
\geq
\Cxx d(p,y)^{-1}-\Cxxi.
\end{equation}
This enables us to estimate~$E(p,y)$.

To utilize the assumption that $P\subset M$ is $\hat\eps$-dense with~$\hat\eps$ chosen by~\eqref{eq:hateps}, suppose that $d(\partial M,p)<\hat\eps$ and let $z_p\in\p M$ be the nearest boundary point to it.
Estimate~\eqref{eq:lambda>1/d} then implies that
\begin{equation}
\label{eq:E<eps-obsolete}
E(p,z_p)
\leq
\frac{\Ci}{\Cxx\hat\eps^{-1}-\Cxxi-\Cii}
\end{equation}
when $\hat\eps<\Cxx/(\Cxxi+\Cii)$.
Due to our choice of~$\hat\eps$, we have
\begin{equation}
\label{eq:rev-cond-1}
\hat\eps
\leq
\frac{\Cxx}{2(\Cxxi+\Cii)}
,
\end{equation}
and so
\begin{equation}
\label{eq:E<eps}
E(p,z_p)
\leq
2\Ci\Cxx^{-1}\hat\eps.
\end{equation}

Consider then an arbitrary point $x\in\p M$.
There is a source point $q\in P$ with $d(x,q)<\hat\eps$. 
Let $z_q\in\partial M$ be the closest boundary point to~$q$.
We have
$
d(z_q,q)
<
\hat \eps,
$
and by bounded geometry (part~\ref{bound:7} of Definition~\ref{def:bdd-geometry})
\begin{equation}
d_{\partial M}(x,z_q)
\leq
\Cviii d_{M}(x,z_q)
\leq
\Cviii(
d_M(x,q)
+
d_M(z_q,q)
)
<
2\Cviii\hat\eps.
\end{equation}
By~\eqref{eq:E<eps} we get
\begin{equation}
E(x)
\leq
E(q,z_q)
+
d_{\partial M}(x,z_q)
\leq
2\Ci\Cxx^{-1}\hat\eps
+
2\Cviii\hat\eps.
\end{equation}
As this bound it is independent of $x\in\p M$, we have the same bound for~$E$ from~\eqref{eq:E3}:
\begin{equation}
\label{eq:E-bound}
E
\leq
2(\Ci\Cxx^{-1}+\Cviii)\hat\eps,
\end{equation}
under the assumption~\eqref{eq:rev-cond-1}.

Combining our choice 
$
\eps_1
=
\Cxix \hat \eps
$
with~\eqref{eq:E-bound}, we get
\begin{equation}
\label{eq:eps2<hateps}
\eps_2
=
\eps_1+E
\leq
(\Cxix+2\Ci\Cxx^{-1}+2\Cviii)
\hat\eps.
\end{equation}
All other estimates for~$\eps_2$ follow from this one.

%
%
%
To satisfy~\eqref{eq:eps2<eps}, we require that
$
\Cxii\eps_2
<
\frac12\eps
$
and
$
\Cxi\sqrt{\eps_2}
<
\frac12\eps
$.
By~\eqref{eq:eps2<hateps} both follow from requiring
\begin{equation}
\Cxii(\Cxix+2\Ci\Cxx^{-1}+2\Cviii)\hat\eps
<
\tfrac12\eps
\end{equation}
and
\begin{equation}
\label{eq:rev-cond-5}
\Cxi\sqrt{(\Cxix+2\Ci\Cxx^{-1}+2\Cviii)\hat\eps}
<
\tfrac12\eps,
\end{equation}
or equivalently
\begin{equation}
\label{eq:rev-cond-4}
\hat\eps
<
\min\left(
\frac{\eps}{2\Cxii(\Cxix+2\Ci\Cxx^{-1}+2\Cviii)}
,
\frac{\eps^2}{4\Cxi^2(\Cxix+2\Ci\Cxx^{-1}+2\Cviii)}
\right).
\end{equation}
Thus~\eqref{eq:rev-cond-1} and~\eqref{eq:rev-cond-4} are satisfies due to~\eqref{eq:hateps}, and by the previous remark the estimate~\eqref{eq:eps2<eps} follows.

It remains to show that each lentil $L=L^{x,y}_{r,d(x,y)-r+\delta}$ with $x,y\in P$ contain source points as claimed.
The thickness of the lentil~$L$ is $\delta=\Cix\eps_2$, so by Lemma~\ref{lma:trans-radius-and-ball} $B(m,\tfrac12\delta)\subset L$, where~$m$ is the midpoint of the lentil~$L$.
Therefore each lentil of thickness~$\delta$ contains a ball of radius $\tfrac12\delta$. 
It follows from $\hat\eps$-density of $P\subset M$ that that each such lentil contains a source point because
\begin{equation}
\label{eq:rev-cond-3}
\hat\eps
=
\Cxix^{-1} \eps_1
=
\tfrac12\Cix\eps_1
\leq
\tfrac12\Cix\eps_2
=
\tfrac12\delta
.
\end{equation}
This is why we chose~$\Cxix$ so that $\Cxix^{-1}=\tfrac12\Cix$.
\end{proof}

\section{Convergence of discrete approximations}
\label{sec:convergence}

\subsection{Deterministic convergence}

With Proposition~\ref{prop:E&L}, Theorem~\ref{thm:C23}, and Theorem~\ref{thm:reverse} the proof of Theorem~\ref{thm:D} is straightforward.

\begin{proof}[Proof of Theorem~\ref{thm:D}]
We begin with Proposition~\ref{prop:separation} with $\Omega=\p M\times(0,T)$.
Out of the parts of graphs we only choose the ones that are full graphs of a function $a_s\colon\p M\to\R$.
It follows from the diameter bound that
\begin{equation}
\tau(s)
\leq
a_s(x)
\leq
\tau(s)+\Civ 
\end{equation}
for all $x\in\p M$. Therefore the set $Q(S,T)$ of~\eqref{eq:Q(S,T)} contains the full graph of~$a_s$ for all the sources $s\in S$ with $0<\tau(s)<T-\Civ$.
Let us denote the set of source points with their complete graphs contained in $\p M\times(0,T)$ by $P_T\subset\pi(S)$.

If~$T$ is too small, we may have $P_T=\emptyset$.
In this case we choose the metric approximation~$M_T$ to be a set of one point and $\alpha_T\colon\p M\to M_T$ the constant map.
When there are sources, we set $M_T=P_T$.

Take any $\eps>0$ and let~$\hat\eps$ be given by~\eqref{eq:hateps}.
As $\bigcup_{T>0}P_T$ is dense in~$M$ by assumption, by compactness there is $T(\eps)>0$ so that~$P_T$ is $\hat\eps$-dense for all $T\geq T(\eps)$.
With the choices $\eps_1=\Cxix\hat\eps$ and $\eps_2=\eps_1+E(T)$, where~$E(T)$ is defined as in~\eqref{eq:E3}, and $\delta=\Cix\eps_2$  we have
\begin{equation}
\label{eq:eps>eps2+sqrt_eps2}
\Cxii\eps_2
+
\Cxi\sqrt{\eps_2}
<
\eps,
\end{equation}
by Theorem~\ref{thm:reverse}.
By the last claim of Theorem~\ref{thm:reverse} the assumption of Theorem~\ref{thm:C23} is satisfied.
Thus the set~$P_T$ is also 
$\eps$-dense
as required in claim (\ref{claim:3}) of Theorem~\ref{thm:C23}. 
We choose a map $\alpha_T\colon\p M\to P_T$ so that~\eqref{eq:alpha-choice} is satisfied when $\alpha=\alpha_T$ and $P=P_T$.
Finally the estimate~\eqref{eq:ghd<eps2+reps2} with~\eqref{eq:eps>eps2+sqrt_eps2} yields
\begin{equation}
\label{eq:ghd<bar-eps1}
\ghd{\partial M}(P_T,\alpha_T;M,\iota)
<
\eps.
\end{equation}
This concludes the the proof.
\end{proof}

\subsection{Dense sources}

Now we are finally able to prove Theorem~\ref{thm:B}.
We will make use of a version of the Myers--Steenrod Theorem~\cite{myers-steenrod,palais-myers-steenrod} which states that a metric isometry between smooth Riemannian manifolds is necessarily smooth.
In the case of simple manifolds it is straightforward to prove and the boundary causes no technical trouble. We record the proof here for the sake of completeness.

\begin{lemma}[Myers--Steenrod Theorem]
\label{lma:myers-steenrod}
If $\Phi\colon M_1\to M_2$ is a metric isometry between simple Riemannian manifolds, it is smooth up to the boundary and an isometry also in the sense that $\Phi^*g_2=g_1$.
\end{lemma}

\begin{proof}
It suffices to show that~$\Phi$ is smooth.
Then preserving distances quickly implies $\Phi^*g_2=g_1$ via differentiating distance functions.

Take any $x_1\in\sisus(M_1)$ and let $x_2=\Phi(x_1)$.
If~$\gamma$ is any constant speed geodesic with $\gamma(0)=x_1$, then $\Phi\circ\gamma$ is a constant speed geodesic starting at~$x_2$ with the same speed.
Let $\Phi_*\colon T_{x_1}M_1\to T_{x_2}M_2$ be the map that maps the velocity vectors of these geodesics to each other.
That is, if $\partial_t\gamma(t)|_{t=0}=v$, then $\partial_t\Phi(\gamma(t))|_{t=0}=\Phi_*(v)$.
We have $\Phi\circ\exp_{x_1}={\exp_{x_2}}\circ\Phi_*$.

The map~$\Phi_\ast$ is clearly homogeneous in positive scalings, maps unit vectors to unit vectors, and is bijective.
By considering reversed geodesics we find $\Phi_*(-v)=-\Phi_*(v)$ for all $v\in T_{x_1}M_1$.

For any two $v,w\in T_{x_1}M_1$ we have~\cite[Equation (2.15)]{Katchalov2001}
\begin{equation}
d(\exp_{x_1}(tv),\exp_{x_1}(tw))^2
=
t^2(
\abs{v}^2+\abs{w}^2-2\ip{v}{w}
)
+
\Order(t^4)
\end{equation}
for small $t>0$.
Due to~$\Phi_*$ preserving the norm, the intertwining property $\Phi\circ\exp_{x_1}={\exp_{x_2}}\circ\Phi_*$, and the isometric nature of~$\Phi$, this implies that
\begin{equation}
\ip{v}{w}
=
\ip{\Phi_*v}{\Phi_*w}
\end{equation}
for all $v,w\in T_{x_1}M_1$.

The map~$\Phi_*$ preserves norms and inner products, and therefore it preserves the (squared) distance between any pair of vectors in~$T_{x_1}M_1$.
Thus~$\Phi_\ast$ is an isometry between finite-dimensional inner product spaces, and as a Euclidean isometry fixing the origin it is linear and thus smooth.
As the exponential maps are diffeomorphisms on their maximal domain of definition, the map $\Phi={{\exp_{x_2}}\circ\Phi_*}\circ\exp_{x_1}^{-1}$ is smooth.
\end{proof}

\begin{proof}[Proof of Theorem~\ref{thm:B}]
%
Let $\iota_1\colon\p M_1\to M_1$ be the usual inclusion map and let $\iota_2=\phi\colon\p M_1\to M_2$ (with extended codomain here for convenience)
so that we have boundary inclusions $\iota_i\colon\p M_1\to M_i$ with the same domain.

For any $T>0$ consider data on the bounded set $\p M\times(-T,T)$.
This time the time interval extends into both the past and the future.
As in the proof of Theorem~\ref{thm:D} above, we get a finite metric space~$M_T$ and a map $\alpha_T\colon\p M_1\to M_T$ so that
\begin{equation}
\ghd{\p M_1}(M_T,\alpha_T;M_i,\iota_i)\to0
\end{equation}
as $T\to\infty$ for both $i=1,2$.
By Proposition~\ref{prop:E&L} the approximating metric space~$M_T$ constructed from data is the same for the two manifolds~$M_1$ and~$M_2$.

We then apply Proposition~\ref{prop:LGH}.
By the triangle inequality
\begin{equation}
\ghd{\p M_1}(M_1,\iota_1;M_2,\iota_2)\to0
\end{equation}
as $T\to\infty$ and so $\ghd{\p M_1}(M_1,\iota_1;M_2,\iota_2)=0$.
Now Proposition~\ref{prop:LGH} gives an isometry $\Phi\colon M_1\to M_2$ so that $\Phi\circ \iota_1=\iota_2$.
The condition $\Phi\circ \iota_1=\iota_2$ simply means that $\Phi|_{\p M_1}=\phi$.
By Lemma~\ref{lma:myers-steenrod} the isometry~$\Phi$ is actually smooth.

It remains to check that the sources $s=(p,t)\in S_1\subset\sisus(M_1)\times\R$ are mapped correctly.
By Theorem~\ref{thm:A} we know that there is a bijection, but we have to verify that it corresponds to the isometry~$\Phi$ as claimed.
In light of Proposition~\ref{prop:separation}, the data can be seen as a collection of graphs.
Let us denote $A_i=\{a_s ; \: s\in S_i\}$.
The two manifolds having equivalent data means that the map $\phi^*\colon A_2\to A_1$ that takes $a\mapsto\phi^*a=a\circ\phi$ is bijective.
We will show that the maps $b_i\colon S_i\to A_i$ with $b_i(s)=a_{s}$ are bijective, and therefore the natural bijection between the sources $S_1\to S_2$ is $\xi=b_2^{-1}\circ(\phi^*)^{-1}\circ b_1$.
We need this map to satisfy~\eqref{eq:xi-Phi}, which now amounts to
\begin{equation}
a_{(p,t)}
=
a_{(\Phi(p),t)}\circ\phi
\end{equation}
as functions $\partial M_1\to\R$ for all $(p,t)\in S_1$.
This follows straightforwardly from the definitions and~$\Phi$ being an isometry.

Let us then show that~$b_i$ is bijective for both~$i$.
To this end, take any two distinct sources $s,\hat s\in S_i$.
The two arrival time functions are
$a_s=r_s+\tau(s)$
and
$a_{\hat s}=r_{\hat s}+\tau(\hat s)$.

If the spatial source points are different, $\pi(s)\neq\pi(\hat s)$, then $r_s-r_{\hat s}$ cannot be a constant function on~$\partial M$ due to Lemma~\ref{lma:Slava+constant}. Therefore~$a_{s}$ and~$a_{\hat s}$ do not coincide.
If the spatial points are the same, then $r_s-r_{\hat s}$ is a constant function~$0$.
As $s\neq\hat s$ but $\pi(s)=\pi(\hat s)$, we must have $\tau(s)\neq\tau(\hat s)$ and thus $a_{s}-a_{\hat s}$ is a non-zero constant function.

We have thus proven that $s\neq\hat s$ implies $b_i(s)\neq b_i(\hat s)$.
This concludes the proof that~$b_i$ is injective and thus the theorem is proven.
\end{proof}

\subsection{Stochastic convergence}

The proof of the stochastic result only requires checking that the point process almost surely produces the correct kind of source set.

\begin{proof}[Proof of Proposition~\ref{prop:poisson}]
Countability of the set $S\subset M\times\R$ given by the homogeneous Poisson point process is certain.
We need to prove that the following properties hold almost surely:
\begin{enumerate}
\item\label{property:1}
$\pi(S)\cap\partial M=\emptyset$.
\item\label{property:2}
$S\subset M\times\R$ is discrete.
\item\label{property:3}
$\pi(S\cap(M\times[0,\infty)))\subset M$ is dense.
\end{enumerate}
The probability that a measurable set $A\subset M\times\R$ contains $k\in\N$ points of the source set~$S$ is
\begin{equation}
\label{eq:Poisson}
P[\#(A\cap S)=k]
=
e^{-\lambda\mu(A)}
\frac{(\lambda\mu(A))^k}{k!}.
\end{equation}
For this and other basic properties of homogeneous Poisson point processes, see e.g.~\cite{Lieshout:point-process,Isham:point-process}.

The first property follows simply from $\mu(\partial M\times\R)=0$.
The second property follows from $A\cap S$ being almost certainly finite when $\mu(A)<\infty$.

For the third property, let $(x_i)_{i=1}^\infty$ be a dense sequence in~$M$ and define for any $i,j\geq1$ the set
\begin{equation}
A_{i,j}
=
B_M(x_i,j^{-1})\times[0,\infty)
\subset
M\times\R.
\end{equation}
Comparability with the natural product measure gives $\mu(A_{i,j})=\infty$.
Therefore it follows from~\eqref{eq:Poisson} that $P[\#(A_{i,j}\cap S)=0]=0$.

The projection $\pi(S\cap(M\times[0,\infty)))$ is dense in~$M$ if $\#(A_{i,j}\cap S)>0$ for all~$i$ and~$j$.
As this event for individual indices has probability~$1$, the event for density is a countable intersection of events of full probability.
Therefore density is almost certain.

The conditions of Theorem~\ref{thm:D} are thus met almost surely.
\end{proof}

\appendix

\section{An improved estimate on transversal radius}
\label{app:transversal}

As mentioned in connection to Lemma~\ref{lma:trans-radius-and-ball}, the estimate can be improved although it is not necessary for the proof of our theorems.
The estimate $R\gtrsim\delta$ of Lemma~\ref{lma:trans-radius-and-ball} can be improved to $R\gtrsim\sqrt{\delta}$ as follows.

\begin{proposition}
\label{prop:transversal-radius}
Consider a simple manifold of bounded geometry with the constants satisfying the condition~\eqref{eq:diameter-curvature-bound}.
Take any $x,y\in M$ and consider the lentil~$L^{x,y}_{r,s}$ with $r,s\in(0,d(x,y))$.
Suppose $\delta^{x,y}_{r,s}=r+s-d(x,y)$ satisfies $\delta^{x,y}_{r,s}\leq\min(r,s)$.
The transversal radius of any lentil satisfies
\begin{equation}
R^{x,y}_{r,s}
>
\min\left(
\Cxvi\sqrt{\delta^{x,y}_{r,s}\min(r,s)},
\tfrac{1}{2}\Civ
\right),
\end{equation}
where the constant is given by~\eqref{eq:c16}.
\end{proposition}

Most of our constants are used for estimates from above, but~$\Cxvi$ is used for estimating from below.
Therefore, unlike most of our constants, it can be made smaller but not larger if needed.
It is also natural that the transversal radius cannot exceed half of the diameter.
Lemma~\ref{lma:trans-radius-and-ball} is more convenient to use and we do not benefit significantly from the improved exponent, so we will not employ Proposition~\ref{prop:transversal-radius} in the proofs of our main results.

\begin{proof}[Proof of Proposition~\ref{prop:transversal-radius}]
Consider a lentil $L=L^{x,y}_{r,s}$ with thickness $\delta=\delta^{x,y}_{r,s}$ and any point $z\in M$ for which the closest point on~$\gamma_{x,y}$ is the midpoint $m=m^{x,y}_{r,s}$.
We want to find conditions on the distance $w\coloneqq d(z,m)$ which ensure that $z\in L$.
We will do so by making explicit estimates on a comparison manifold of constant sectional curvature and then translating the resulting estimate to the actual manifold.

We assume that $w\leq\tfrac{1}{2}\Civ$.
Due to~\eqref{eq:diameter-curvature-bound} this implies
$
w^2\Cvi
<
\pi^2/4
$,
and so $1-\frac12w^2\Cvi\in(-1,1]$.
This will ensure that the argument of $\arccos\colon[-1,1]\to[0,\pi]$ stays within the domain in the following treatment.
We may freely assume $w>0$ when convenient, as the case $w=0$ can be given a trivial separate treatment.

Let us denote $R\coloneqq d(x,m)$ and $d\coloneqq d(x,z)$.
The points~$x$, $z$, and~$m$ form a triangle with a right angle at~$m$.
Consider the corresponding triangle on the~$M_{\Cvi}$ with constant sectional curvature~$\Cvi$ with with the right angle at~$\tilde m$ with the same lengths~$R$ and~$w$ of the catheti.
The length of the hypotenuse is
\begin{equation}
\tilde d
=
\Cvi^{-1/2}
\arccos[\cos(R\sqrt{\Cvi})\cos(w\sqrt{\Cvi})]
. 
\end{equation}
The cosine satisfies $\cos(x)\geq1-\frac12x^2$, so
\begin{equation}
\tilde d
\leq
\Cvi^{-1/2}
\arccos[\cos(R\sqrt{\Cvi})[1-\tfrac12w^2\Cvi]].
\end{equation}
The function $t\mapsto\arccos(1-t^2)$ is Lipschitz-continuous on $[0,T]$
with the Lipschitz constant
\begin{equation}
\frac{2}{\sqrt{2-T^2}}
\end{equation}
whenever $T<\sqrt{2}$.
In our setting
$T=\sqrt{\Cxxx}$
(with the constant of~\eqref{eq:c30})
and the two values of~$t$ where we compare $\arccos(1-t^2)$ are
\begin{equation}
\sqrt{1-\cos(R\sqrt{\Cvi})[1-\tfrac12w^2\Cvi]}
\end{equation}
and
\begin{equation}
\sqrt{1-\cos(R\sqrt{\Cvi})}.
\end{equation}
The condition~\eqref{eq:diameter-curvature-bound} implies that indeed $T< \sqrt{2}$, so the Lipschitz constant is~$\Cxxiii$ as given in~\eqref{eq:c23}.
This Lipschitz-continuity now gives
\begin{equation}
\tilde d
\leq
R
+
\Cvi^{-1/2}\Cxxiii
\left[
\sqrt{1-\cos(R\sqrt{\Cvi})[1-\tfrac12w^2\Cvi]}
-
\sqrt{1-\cos(R\sqrt{\Cvi})}
\right]
.
\end{equation}
Using the estimate $\sqrt{1-a+ab}\leq\sqrt{1-a}+\frac12\abs{ab}\sqrt{1-a}$ (which is valid for $a<1$ and $b\in\R$) leads to
\begin{equation}
\tilde d
\leq
R
+
\Cvi^{-1/2}\Cxxiii
\frac{\Cvi\abs{\cos(R\sqrt{\Cvi})}}{4\sqrt{1-\cos(R\sqrt{\Cvi})}}
w^2
.
\end{equation}
A simple calculation gives $\abs{\cos(x)}[1-\cos(x)]^{-1/2}\leq4/x$ for all $x\in(0,\pi)$.
Thus we find
\begin{equation}
\tilde d
\leq
R
+
\Cxxiii R^{-1}w^2,
\end{equation}
completing our estimate of the comparison length~$\tilde d$.

By Lemma~\ref{lma:curvature-comparison2} we have $d\leq\tilde d$, and so
\begin{equation}
\label{eq:d<R+w}
d
\leq
R+\Cxxiii R^{-1}w^2.
\end{equation}
Let us use this to see when $z\in L$.

To ensure $z\in L$, we need $d<r$.
As $r=R+\frac12\delta$, estimate~\eqref{eq:d<R+w} says that $d<r$ whenever $\Cxxiii R^{-1}w^2<\frac12\delta$.
Thus we get the condition $w^2<(2\Cxxiii)^{-1}\delta(r-\frac12\delta)$ and the rougher bound $\tfrac{1}{2}\Civ$.
A similar treatment of the condition $d(y,z)<s$ leads us to $w^2<(2\Cxxiii)^{-1}\delta(s-\frac12\delta)$.
We have thus shown that
\begin{equation}
R^{x,y}_{r,s}
>
\min\left(
(2\Cxxiii)^{-1/2}
\sqrt{\delta(r-\tfrac12\delta)}
,
(2\Cxxiii)^{-1/2}
\sqrt{\delta(s-\tfrac12\delta)}
,
\tfrac{1}{2}\Civ
\right).
\end{equation}
As $\delta<r,s$, this implies the claimed inequality.
\end{proof}

\section{Constants}
\label{app:const}

The constants in our results depend explicitly on the constants describing the simplicity of the manifold in Definition~\ref{def:bdd-geometry}.
The same applies to a number of auxiliary constants appearing in our lemmas and their proofs.
To express these constants neatly, we make use of the three constants
\begin{align}
\label{eq:c28}
\Cxxviii
&=
\operatorname{sinc}(\Civ\sqrt{\Cvi})
,
\\
\label{eq:c13}
\Cxiii
&=
\operatorname{sinhc}(\Civ\sqrt{\Cv})
,
\\
\label{eq:c30}
\Cxxx
&=
\operatorname{vercos}(\Civ\sqrt{\Cvi}).
\end{align}
defined in terms of three less common trigonometric functions:
$\operatorname{sinc}(x)=\sin(x)/x$,
its hyperbolic counterpart $\operatorname{sinhc}(x)=\sinh(x)/x$,
and the vercosine $\operatorname{vercos}(x)=1+\cos(x)$.

We will express our constants below in two forms.
The first one is neatly expressed in terms of the fundamental constants and those of~\eqref{eq:c28}, \eqref{eq:c13}, and~\eqref{eq:c30}.
The second one is the one that arises naturally from the proof.
We recommend using the first ones for reading the claims and the second ones for reading the proofs.
Our main results use these constants:
\begin{align}
\label{eq:c9}
\Cix
&=
2\Cvii\Cxiii
\\&=
2\Cxiv,
\\
\label{eq:c10}
\Cx
&=
2+5\Cvii\Cxiii
\\&=
2+\Cix+\Cxvii,
\\
\label{eq:c11}
\Cxi
&=
2\sqrt{10}\Civ^{1/2}\Cvii^{1/2}\Cxxviii^{1/2}\Cxiii^{3/2}
\\&=
\Cxv\sqrt{\Cix},
\\
\label{eq:c12}
\Cxii
&=
3+5\Cvii\Cxiii
\\&=
\Cx+1,
\\
\label{eq:c19}
\Cxix
&=
\Cvii^{-1}\Cxiii^{-1}
\\&=
2\Cix^{-1}
,
\\
\label{eq:c25}
\Cxxv
&=
\frac{\Cxx}{2(\Cxxi+\Cii)}
,
\\
\label{eq:c26}
\Cxxvi
&=
\frac{1}{2(3+5\Cvii\Cxiii)(\Cvii^{-1}\Cxiii^{-1}+2\Ci\Cxx^{-1}+2\Cviii)}
\\&=
\frac{1}{2\Cxii(\Cxix+2\Ci\Cxx^{-1}+2\Cviii)}
,
\\
\label{eq:c27}
\Cxxvii
&=
\frac{1}{160\Civ\Cvii\Cxxviii\Cxiii^3(\Cvii^{-1}\Cxiii^{-1}+2\Ci\Cxx^{-1}+2\Cviii)}
\\&=
\frac{1}{4\Cxi^2(\Cxix+2\Ci\Cxx^{-1}+2\Cviii)}
.
\end{align}

Our auxiliary results make use of these constants:
\begin{align}
\label{eq:c14}
\Cxiv
&=
\Cvii\Cxiii
,
\\
\label{eq:c15}
\Cxv
&=
2\sqrt{5}\Cxxviii^{1/2}\Cxiii\Civ^{1/2}
,
\\
\label{eq:c22}
\Cxxii
&=
\Cxxviii^{-1}
,
\\
\label{eq:c17}
\Cxvii
&=
3\Cxiv
.
\end{align}

Finally, these constants appear in Proposition~\ref{prop:transversal-radius} in Appendix~\ref{app:transversal}:
\begin{align}
\label{eq:c16}
\Cxvi
&=
2^{-3/2}
\Cxxx^{1/4}
,
\\
\label{eq:c23}
\Cxxiii
&=
2\Cxxx^{-1/2}
.
\end{align}

\bibliographystyle{abbrv}
\bibliography{bibliography,introduction}

\end{document}